\documentclass[12pt,reqno]{amsart}%
\usepackage{amsmath, amsfonts, amssymb, amsthm, amscd, amsbsy}
\usepackage{fancyhdr}
\usepackage[usenames,dvipsnames,svgnames,x11names,hyperref]{xcolor}
\usepackage{geometry}
\usepackage{graphicx}
\usepackage{enumerate}
\usepackage[pagebackref]{hyperref}
\usepackage{amsmath}
\usepackage{amsfonts}
\usepackage{amssymb}%
\providecommand{\U}[1]{\protect\rule{.1in}{.1in}}
\hypersetup{backref=true,
pagebackref=true,
hyperindex=true,
colorlinks=true,
breaklinks=true,
urlcolor=NavyBlue,
linkcolor=Fuchsia,
bookmarks=true,
bookmarksopen=false,
filecolor=black,
citecolor=ForestGreen,
linkbordercolor=red
}
\newtheorem{theorem}{Theorem}[section]
\theoremstyle{plain}

\newtheorem{corollary}{Corollary}[section]

\numberwithin{equation}{section}
\allowdisplaybreaks
\AtBeginDocument{\hypersetup{pdfborder={0 0 0.01}}}

\begin{document}
\title[CKN inequalities]{Caffarelli-Kohn-Nirenberg inequalities for curl-free vector fields and second
order derivatives}
\author{Cristian Cazacu}
\address{Cristian Cazacu: $^{1}$Faculty of Mathematics and Computer Science \\
University of Bucha-rest\\
010014 Bucharest, Romania\\
\& $^{2}$Gheorghe Mihoc-Caius Iacob Institute of Mathematical\\
Statistics and Applied Mathematics of the Romanian Academy\\
050711 Bucharest, Romania }
\email{cristian.cazacu@fmi.unibuc.ro}
\author{Joshua Flynn}
\address{Joshua Flynn: Department of Mathematics\\
University of Connecticut\\
Storrs, CT 06269, USA}
\email{joshua.flynn@uconn.edu}
\author{Nguyen Lam}
\address{Nguyen Lam: School of Science and the Environment\\
Grenfell Campus, Memorial University of Newfoundland\\
Corner Brook, NL A2H5G4, Canada }
\email{nlam@grenfell.mun.ca}
\thanks{J.F. was partially supported by a Simons Collaboration grant from the Simons
Foundation. N.L. was partially supported by an NSERC Discovery Grant.}
\date{\today}

\begin{abstract}
The present work has as a first goal to extend the previous results in
\cite{CFL20} to weighted uncertainty principles with nontrivial radially
symmetric weights applied to curl-free vector fields. Part of these new
inequalities generalize the family of Caffarelli-Kohn-Nirenberg (CKN)
inequalities studied by Catrina and Costa in \cite{CC} from scalar fields to
curl-free vector fields. We will apply a new representation of curl-free
vector fields developed by Hamamoto in \cite{HT21}. The newly obtained results
are also sharp and minimizers are completely described.

Secondly, we prove new sharp second order interpolation functional
inequalities for scalar fields with radial weights generalizing the previous
results in \cite{CFL20}. We apply new factorization methods being inspired by
our recent work \cite{CFL21}. The main novelty in this case is that
we are able to find a new independent family of minimizers based on the
solutions of Kummer's differential equations.

We point out that the two types of weighted inequalities under
consideration (first order inequalities for curl-free vector fields vs. second
order inequalities for scalar fields) represent independent families of
inequalities unless the weights are trivial.

\end{abstract}
\subjclass[2020]{81S07, 26D10, 46E35, 26D15, 58A10}
\keywords{uncertainty principles, Caffarelli-Kohn-Nirenberg inequalities, curl-free
vector fields, sharp constants, minimizers}
\maketitle

\section{Introduction}

The aim of this paper is twofold:

\begin{enumerate}
[a).]

\item We study the sharp constants in the weighted inequalities of
Caffarelli-Kohn-Nirenberg (CKN) type having the form
\begin{equation}
\label{type1}\int_{\mathbb{R}^{N}}\frac{| \mathbf{U}|^{2}}{|x|^{2a}%
}\mathrm{dx}\int_{\mathbb{R}^{N}}\frac{|\nabla\mathbf{U}|^{2}}{|x|^{2b}%
}\mathrm{dx}\geq C_{1}(N,a,b)\left(  \int_{\mathbb{R}^{N}}\frac{|\mathbf{U}%
|^{2}}{|x|^{a+b+1}}\mathrm{dx}\right)  ^{2}\text{, }
\end{equation}
for any curl-free vector field (defined below) $\mathbf{U}\in\left(  C_{0}^{\infty}(\mathbb{R}^{N}\setminus\{0\})\right)^{N}$
and $C_{1}(N, a, b)$ is the sharp constant in \eqref{type1}; the
parameter $(a,b)$ may be any point in $\mathbb{R}^{2}$; $N\geq2$ and
$C_{0}^{\infty}(\mathbb{R}^{N}\setminus\{0\})$ denotes the space of smooth
functions compactly supported in $\mathbb{R}^{N}\setminus\{0\}$.

\item We analyze the sharp second order CKN inequalities for scalar fields of
the form
\begin{equation}
\int_{\mathbb{R}^{N}}\frac{|\Delta u|^{2}}{|x|^{2a}}\mathrm{dx}\int
_{\mathbb{R}^{N}}|x|^{2a+2}\left\vert \frac{x}{\left\vert x\right\vert }%
\cdot\nabla u\right\vert ^{2}\mathrm{dx}\geq C_{2}(N,a)\left(  \int
_{\mathbb{R}^{N}}|\nabla u|^{2}\mathrm{dx}\right)  ^{2}, \label{type2}%
\end{equation}
for any $u\in C_{0}^{\infty}(\mathbb{R}^{N}\setminus\{0\})$; $a\in\mathbb{R}$
is given and $C_{2}(N,a)$ denotes the best constant in \eqref{type2}.
\end{enumerate}

\paragraph{\textbf{State of the art.}}

The CKN inequalities were first introduced in 1984 by Caffarelli, Kohn and
Nirenberg in the pioneering work \cite{CKN} to generalize many well-known and
important inequalities in analysis such as Gagliardo-Nirenberg inequalities,
Hardy-Sobolev inequalities, Nash's inequalities, Sobolev inequalities, etc.
Since then, due to their important roles and applications in many areas of
pure and applied mathematics, especially in analysis and PDE, the CKN type
inequalities have been studied extensively in the literature by many authors;
e.g., see \cite{BM18, BSY21, CKN2, CW, CLZ21, DD02, DEL16, D18, DLL18, F20, FLL21, LL} and the
references therein. Afterwards, many researchers have become interested in
studying finer properties such as sharp constants for CKN inequalities and
their extremizers.
An important subfamily consists of the extensively studied sharp $L^{2}$-CKN inequalities given by
\begin{equation}
\int_{\mathbb{R}^{N}}\frac{|u|^{2}}{|x|^{2a}}\mathrm{dx}\int_{\mathbb{R}^{N}%
}\frac{|\nabla u|^{2}}{|x|^{2b}}\mathrm{dx}\geq C^{2}(N,a,b)\left(
\int_{\mathbb{R}^{N}}\frac{|u|^{2}}{|x|^{a+b+1}}\mathrm{dx}\right)
^{2}\text{, }u\in C_{0}^{\infty}(\mathbb{R}^{N}\setminus\{0\}).\label{CKN}%
\end{equation}
Notice that inequality \eqref{CKN} represents the scalar version of the family
\eqref{type1}. The best constant $C^{2}(N,a,b)>0$ is known and the minimizers
are fully described. We recall that these aspects of \eqref{CKN} were first
studied by Costa in \cite{Co} for a particular range of parameters by using
the \textit{expanding-the-square} method, and then by Catrina and Costa in
\cite{CC} for the full range of parameters using spherical harmonics
decomposition and the Kelvin transform. (An argument without needing spherical harmonics is found in \cite{F20}.)
The obtained results depend on some
parameter regions defined as in
\begin{equation}
\left\{
\begin{array}
[c]{ll}%
\mathcal{A}_{1}:=\{(a,b)\ |\ b+1-a>0,\ b\leq(N-2)/2\} & \\[5pt]%
\mathcal{A}_{2}:=\{(a,b)\ |\ b+1-a<0,\ b\geq(N-2)/2\} & \\[5pt]%
\mathcal{A}:=\mathcal{A}_{1}\cup\mathcal{A}_{2} & \\[5pt]%
\mathcal{B}_{1}:=\{(a,b)\ |\ b+1-a<0,\ b\leq(N-2)/2\} & \\[5pt]%
\mathcal{B}_{2}:=\{(a,b)\ |\ b+1-a>0,\ b\geq(N-2)/2\} & \\[5pt]%
\mathcal{B}:=\mathcal{B}_{1}\cup\mathcal{B}_{2} &
\end{array}.
\right.
\end{equation}
More precisely, it was showed in \cite{CC} that in the region $\mathcal{A}$,
the best constant is $\displaystyle C(N,a,b)=\frac{|N-(a+b+1)|}{2}$ and it is
achieved by the nontrivial functions $\displaystyle u(x)=\gamma\exp\left(
\frac{t|x|^{b+1-a}}{b+1-a}\right)  $, with $t<0$ in $\mathcal{A}_{1}$ and
$t>0$ in $\mathcal{A}_{2}$, and $\gamma$ a nonzero constant. In the region
$\mathcal{B}$, the best constant is $\displaystyle C(N,a,b)=\frac
{|N-(3b-a+3)|}{2}$ and it is achieved by the functions $\displaystyle
u(x)=\gamma|x|^{2(b+1)-N}\exp\left(  \frac{t|x|^{b+1-a}}{b+1-a}\right)  $,
with $t>0$ in $\mathcal{B}_{1}$ and $t<0$ in $\mathcal{B}_{2}$. In addition,
the only values of the parameters where the best constant is not achieved are
those on the line $a=b+1$ where the best constant is $\displaystyle
C(N,b+1,b)=\frac{|N-2(b+1)|}{2}$. In this latter case the CKN inequality
degenerates into a weighted Hardy-Leray type inequality.

Very recently, the authors of the present paper provided in \cite{CFL21} a
very simple and direct proof (bypassing spherical harmonics and the Kelvin transform) of the refined CKN inequality
\begin{equation}
\int_{\mathbb{R}^{N}}\frac{|u|^{2}}{|x|^{2a}}\mathrm{dx}\int_{\mathbb{R}^{N}%
}\frac{|x\cdot\nabla u|^{2}}{|x|^{2b+2}}\mathrm{dx}\geq\tilde{C}%
^{2}(N,a,b)\left(  \int_{\mathbb{R}^{N}}\frac{|u|^{2}}{|x|^{a+b+1}}%
\mathrm{dx}\right)  ^{2}\text{, }u\in C_{0}^{\infty}(\mathbb{R}^{N}%
\setminus\{0\}),\label{CKN_refined}%
\end{equation}
where $\tilde{C}^{2}(N,a,b)$ denotes the sharp constant in
\eqref{CKN_refined}. It was shown in \cite{CFL21} that $\tilde{C}%
^{2}(N,a,b)=C^{2}(N,a,b)$ for the whole range of parameters $(a,b)\in
\mathbb{R}^{2}$. However, inequality \eqref{CKN_refined} requires only the
radial derivative $\partial_{r}:=\frac{x}{|x|}\cdot\nabla$ on the left hand
side instead of the full gradient. Of course, this makes \eqref{CKN_refined}
finer than \eqref{CKN} since $\left\vert \partial_{r}u\right\vert
\leq\left\vert \nabla u\right\vert $. We also characterized in \cite{CFL21}
all the optimizers. The interesting fact which occurs is that the minimizers
of \eqref{CKN_refined} are not necessarily radially symmetric (as happens for
\eqref{CKN}) and they differ from the minimizers of \eqref{CKN} by a
multiplicative function depending only on the spherical component. It is
important to emphasize that the obtained minimizers do not belong to the space
$C_{0}^{\infty}(\mathbb{R}^{N}\setminus\{0\})$ but to the functional spaces
defined as the closure of $C_{0}^{\infty}(\mathbb{R}^{N}\setminus\{0\})$ in
the corresponding energy norm. This was an aspect which is not considered in
\cite{CC} but we explained it in detail in \cite{CFL21}.

Going back to the subclass \eqref{CKN} of CKN inequalities, we notice that it
recovers the mathematical formulation for three famous inequalities of quantum
mechanics, namely the Heisenberg Uncertainty Principle (HUP), the Hydrogen
Uncertainty Principle (HyUP) and the Hardy Inequality (HI). The HUP is obtained
when $a=-1$ and $b=0$: for any $N\geq1$, there holds
\begin{equation}
\int_{\mathbb{R}^{N}}|\nabla u|^{2}\mathrm{dx}\int_{\mathbb{R}^{N}}%
|x|^{2}|u|^{2}\mathrm{dx}\geq\frac{N^{2}}{4}\left(  \int_{\mathbb{R}^{N}%
}|u|^{2}\mathrm{dx}\right)  ^{2}, \quad\forall u\in C_{0}^{\infty}%
(\mathbb{R}^{N}). \label{HUP}%
\end{equation}
It can also be extended to functions $u$ in the Schwartz space $\mathcal{S}%
(\mathbb{R}^{N})$ or in appropriate Sobolev spaces. It can also be verified
that the constant $\displaystyle \frac{N^{2}}{4}$ in the \eqref{HUP} is sharp and
is attainable in $\mathcal{S}(\mathbb{R}^{N})$ by minimizers of the form
$\displaystyle u(x)=\gamma e^{-\beta|x|^{2}}$, $\gamma\in\mathbb{R}$,
$\beta>0$ (see, e.g. \cite{FS97}). When $a=b=0$ we recover the HyUP: for any
$N\geq2$, there holds
\begin{equation}
\int_{\mathbb{R}^{N}}|\nabla u|^{2}\mathrm{dx}\int_{\mathbb{R}^{N}}%
|u|^{2}\mathrm{dx}\geq\frac{\left(  N-1\right)  ^{2}}{4}\left(  \int
_{\mathbb{R}^{N}}\frac{|u|^{2}}{\left\vert x\right\vert }\mathrm{dx}\right)
^{2}, u\in C_{0}^{\infty}(\mathbb{R}^{N}). \label{HyUP}%
\end{equation}
The constant $\displaystyle \frac{\left(  N-1\right)  ^{2}}{4}$ in
\eqref{HyUP} is also optimal and can be attained by minimizers of the form
$\displaystyle u(x)=\gamma e^{-\beta|x|}$, $\gamma\in\mathbb{R}$, $\beta>0$
(see, e.g. \cite{Fra11}). Notice that in this case these minimizers are not in
$\mathcal{S}(\mathbb{R}^{N})$ but are in the Sobolev space $W^{1,2}%
(\mathbb{R}^{N})$.

In the case $a=1$ and $b=0$ we have a degenerate case of \eqref{CKN} which
emerges as the famous HI: for any $N\geq3$, there holds
\begin{equation}
\int_{\mathbb{R}^{N}}|\nabla u|^{2}\mathrm{dx}\geq\frac{\left(  N-2\right)
^{2}}{4}\int_{\mathbb{R}^{N}}\frac{|u|^{2}}{\left\vert x\right\vert ^{2}%
}\mathrm{dx}, \quad\forall u\in C_{0}^{\infty}(\mathbb{R}^{N}). \label{Har}%
\end{equation}
The constant in \eqref{Har} is sharp but not achieved.
There has been considerable interest in the HI, its extension and their applications over the last few decades.
One of the very first application of the HI \eqref{Har} appeared for
$N=3$ in Leray's paper \cite{Ler33} when he studied the Navier-Stokes
equations. Since the HI is not a priority of this paper we avoid to cite
futher references for \eqref{Har}. Inequalities \eqref{HUP}-\eqref{Har} are
independent but all of them can be deduced by applying the divergence theorem
and Holder inequality. On the other hand, it is easy to check that the HI
\eqref{Har} implies \eqref{HUP} with a worse constant, i.e. $\frac{(N-2)^{2}%
}{4}$ instead of $\frac{N^{2}}{4}$.

\paragraph{\textbf{HUP, HyUP and HI for vector fields}}

Motivated by questions in hydrodynamics it is important and interesting to
compute the best constants of functional inequalities such as HI, HUP,
CKN-type inequalities, etc., for vector fields. Special classes of vector
fields which appear frequently in applications are divergence-free vector
fields (a restriction which enhances for instance the Stokes/Navier-Stokes
equations for incompressible fluids) or curl-free vector fields (see e.g.
Maxwell equations of electromagnetism). Here a divergence-free
vector field $\mathbf{U}=\left(  U_{1},...,U_{N}\right)  \in\left(
C_{0}^{\infty}(\mathbb{R}^{N})\right)  ^{N}$ satisfies $\mathrm{\operatorname{div}%
}\mathbf{U}=0$. By a curl-free vector field we understand a vector field of the form
$\mathbf{U}=\nabla u$, where $u: \mathbb{R}^{N}\mapsto\mathbb{C}$ is a scalar
potential field. As a consequence of Cauchy-Schwarz inequality, HUP, HyUP, HI,
$L^{2}$-CKN inequalities for scalar fields transfer easily with the same best
constant to non-restricted vector fields. A key question which arises is
whether the new sharp constant remains the same as in the scalar case or
improves in the case of vector fields with restrictions. Let us next recall
the main results which have been done so far in this direction. If there is no
restriction on the vector field, then we get from (\ref{Har}) that
\begin{equation}
\int_{\mathbb{R}^{N}}|\nabla\mathbf{U}|^{2}\mathrm{dx}\geq\frac{\left(
N-2\right)  ^{2}}{4}\int_{\mathbb{R}^{N}}\frac{|\mathbf{U}|^{2}}{\left\vert
x\right\vert ^{2}}\mathrm{dx} \label{vHar}%
\end{equation}
and the constant $\displaystyle \frac{\left(  N-2\right)  ^{2}}{4}$ is
optimal. In \cite{CM08}, Costin and Maz'ya showed that for divergence-free
vector fields with some additional restrictions (axisymmetric assumption) we
have
\[
\int_{\mathbb{R}^{N}}|\nabla\mathbf{U}|^{2}\mathrm{dx}\geq\frac{\left(
N-2\right)  ^{2}}{4}\left(  1+\frac{8}{N^{2}+4N-4}\right)  \int_{\mathbb{R}%
^{N}}\frac{|\mathbf{U}|^{2}}{\left\vert x\right\vert ^{2}}\mathrm{dx}%
\]
and $\displaystyle\frac{\left(  N-2\right)  ^{2}}{4}\left(  1+\frac{8}%
{N^{2}+4N-4}\right)  $ is sharp. Obviously, this new sharp constant improves
the original optimal constant $\displaystyle\frac{\left(  N-2\right)  ^{2}}%
{4}$ of the Hardy inequality (\ref{vHar}) without restrictions. Recently,
Hamamoto proved in \cite{Ham20} that the additional assumption in \cite{CM08}
that $\mathbf{U}$ is axisymmetric can be removed and we still achieve the same
best constant $\displaystyle\frac{\left(  N-2\right)  ^{2}}{4}\left(
1+\frac{8}{N^{2}+4N-4}\right)  $. The HI for curl-free vector fields is
equivalent to the so-called Hard-Rellich inequality
\[
\int_{\mathbb{R}^{N}}|\Delta u|^{2}\mathrm{dx}\geq\lambda^{\sharp}(N)
\int_{\mathbb{R}^{N}}\frac{|\nabla u|^{2}}{\left\vert x\right\vert ^{2}%
}\mathrm{dx,}%
\]
where $u$ si a potential scalar fields for $\mathbf{U}$, i.e. $\mathbf{U}%
=\nabla u$. The sharp constant $\lambda^{\sharp}(N)$ was determined
progressively depending on the dimension $N$. First it was shown in \cite{TZ}
that $\lambda^{\sharp}(N)=\frac{N^{2}}{4}$ for any $N\geq5$ by using spherical
harmonics decomposition. Later, for $N\in\{3, 4\}$ it was shown that
$\lambda^{\sharp}(3)=\frac{25}{38}$ and $\lambda^{\sharp}(4)=3$ by means of
the Fourier transform (in \cite{B}), Bessel pairs (in \cite{GM}) or an alternative
proof using spherical harmonics decomposition in \cite{C}. Hamamoto and
Takahashi also investigated and determined in \cite{HT19} the sharp constants
of weighted Hardy type inequalities for curl-free vector fields.

In contrast with HI, to our knowledge HUP, HyUP or CKN inequalities for
divergence-free and curl-free vector fields with best constants have been less
investigated so far.

The problem of finding the best constant of HUP when one replaces $u$ in
\eqref{HUP} by a divergence-free vector field $\mathbf{U}$ was posed by Maz'ya
in \cite[Section 3.9]{M}. Indeed, Maz'ya raised the following question in
\cite[Section 3.9]{M}: determine the best constant $\mu^{\ast}(N)$ in the
following inequality
\begin{equation}
\int_{\mathbb{R}^{N}}|\nabla\mathbf{U}|^{2}\mathrm{dx}\int_{\mathbb{R}^{N}%
}|x|^{2}|\mathbf{U}|^{2}\mathrm{dx}\geq\mu^{\ast}(N)\left(  \int
_{\mathbb{R}^{N}}|\mathbf{U}|^{2}\mathrm{dx}\right)  ^{2},\quad\forall
\mathbf{U}\in\left(  C_{0}^{\infty}(\mathbb{R}^{N})\right)  ^{N}%
,\quad\mathrm{\operatorname{div}}\mathbf{U}=0. \label{HUP_div}%
\end{equation}
In the recent paper \cite{CFL20}, among other obtained results we answered in
particular to the Maz'ya question for $N=2$ and we showed that $\mu^{\ast
}(2)=4$. Motivated by the fact that in $\mathbb{R}^{2}$, the divergence-free
vector fields are isometrically isomorphic to the curl-free vector fields in
the two-dimensional case, a divergence-free vector can be written in the form
$\mathbf{U}=(-u_{x_{2}},u_{x_{1}})$ where $u$ is a scalar field. If
$\mathbf{U}\in\left(  C_{0}^{\infty}(\mathbb{R}^{2})\right)  ^{2}$ then also
$u\in C_{0}^{\infty}(\mathbb{R}^{2})$. Then, after integration by parts we get%

\[
\int_{\mathbb{R}^{2}}|\nabla\mathbf{U}|^{2}\mathrm{dx}=\int_{\mathbb{R}^{2}%
}|\Delta u|^{2}\mathrm{dx}.
\]
Therefore \eqref{HUP_div} is equivalent to the following second order CKN-type
inequality in $\mathbb{R}^{2}$:%

\begin{equation}
\int_{\mathbb{R}^{2}}|\Delta u|^{2}\mathrm{dx}\int_{\mathbb{R}^{2}}%
|x|^{2}|\nabla u|^{2}\mathrm{dx}\geq\mu^{\star}(2)\left(  \int_{\mathbb{R}%
^{2}}|\nabla u|^{2}\mathrm{dx}\right)  ^{2}. \label{HUP_div_scalar}%
\end{equation}
Then, by using spherical harmonic decomposition, we showed in \cite{CFL20}
that $\mu^{\star}(2)=4$ is sharp in (\ref{HUP_div_scalar}) and is achieved by
the Gaussian profiles of the form $\displaystyle u(x)=\alpha e^{-\beta|x|^{2}%
}$, $\beta>0$. Therefore $\mu^{\star}(2)=4$ is sharp in (\ref{HUP_div}) and is
attained by the vector fields of the form $\displaystyle \mathbf{U}\left(
x\right)  =\left(  -\alpha e^{-\beta|x|^{2}}x_{2},\alpha e^{-\beta|x|^{2}%
}x_{1}\right)  $, $\beta>0$, $\alpha\in\mathbb{R}$.

Very recently, Hamamoto answered Maz'ya's open question in the remaining case
$N\geq3$ in \cite{Ham21}. More precisely, Hamamoto applied the
poloidal-toroidal decomposition to prove that $\mu^{\ast}(N)=\frac{1}%
{4}\left(  \sqrt{N^{2}-4\left(  N-3\right)  }+2\right)  ^{2}$ when $N\geq3$.

It is worthy to mention that we also proved in \cite{CFL20} the following
sharp HUP for curl-free vector fields: for $\mathbf{U}\in\left(  C_{0}%
^{\infty}(\mathbb{R}^{N})\right)  ^{N}$, $\mathrm{\operatorname{curl}%
}\mathbf{U}=0$, there holds
\begin{equation}
\int_{\mathbb{R}^{N}}|\nabla\mathbf{U}|^{2}\mathrm{dx}\int_{\mathbb{R}^{N}%
}|x|^{2}|\mathbf{U}|^{2}\mathrm{dx}\geq\left(  \frac{N+2}{2}\right)
^{2}\left(  \int_{\mathbb{R}^{N}}|\mathbf{U}|^{2}\mathrm{dx}\right)  ^{2},
\label{HUP_curl}%
\end{equation}
by improving the best constant $\frac{N^{2}}{4}$ which corresponds to scalar
fields in \eqref{HUP}. Indeed, since in this case we can write $\mathbf{U}%
=\nabla u$ for some scalar potential $u:\mathbb{R}^{N}\mapsto\mathbb{C}$,
(\ref{HUP_curl}) is equivalent to the following second order CKN inequality:
\begin{equation}
\int_{\mathbb{R}^{N}}|\Delta u|^{2}\mathrm{dx}\int_{\mathbb{R}^{N}}%
|x|^{2}|\nabla u|^{2}\mathrm{dx}\geq\frac{(N+2)^{2}}{4}\left(  \int
_{\mathbb{R}^{N}}|\nabla u|^{2}\mathrm{dx}\right)  ^{2}.
\label{HUP_curl_scalar}%
\end{equation}
Then, by using spherical harmonics decomposition, we proved in \cite{CFL20}
that the constant $\displaystyle \frac{(N+2)^{2}}{4}$ is optimal in
\eqref{HUP_curl_scalar} and is attained for Gaussian profiles of the form
$\displaystyle u(x)=\alpha e^{-\beta|x|^{2}}$, $\beta>0$. Therefore, the
constant $\displaystyle \left(  \frac{N+2}{2}\right)  ^{2}$ is sharp in
(\ref{HUP_curl}) and is achieved by the vector fields
$\displaystyle \mathbf{U}\left(  x\right)  =\alpha e^{-\beta|x|^{2}}x$,
$\beta>0$.

Along the same line of thought, we proved in addition in \cite{CFL20}\ that
for $N\geq5$, then for $\mathbf{U}\in\left(  C_{0}^{\infty}(\mathbb{R}%
^{N})\right)  ^{N}$, $\mathrm{\operatorname{curl}}\mathbf{U}=0:$
\begin{equation}
\int_{\mathbb{R}^{N}}|\nabla\mathbf{U}|^{2}\mathrm{dx}\int_{\mathbb{R}^{N}%
}|\mathbf{U}|^{2}\mathrm{dx}\geq\frac{(N+1)^{2}}{4}\left(  \int_{\mathbb{R}%
^{N}}\frac{|\mathbf{U}|^{2}}{|x|}\mathrm{dx}\right)  ^{2}. \label{HyUP_curl}%
\end{equation}
Equivalently, that is
\begin{equation}
\int_{\mathbb{R}^{N}}|\Delta u|^{2}\mathrm{dx}\int_{\mathbb{R}^{N}}|\nabla
u|^{2}\mathrm{dx}\geq\frac{(N+1)^{2}}{4}\left(  \int_{\mathbb{R}^{N}}%
\frac{|\nabla u|^{2}}{\left\vert x\right\vert }\mathrm{dx}\right)  ^{2}.
\label{HyUP_curl_scalar}%
\end{equation}
for a scalar potential $u$ corresponding to $\mathbf{U}$. The constant
$\displaystyle \frac{(N+1)^{2}}{4}$ is optimal and it is attained for
functions of the form $\displaystyle u(x)=\alpha\left(  1+\beta\left\vert
x\right\vert \right)  e^{-\beta|x|}$, $\beta>0$. Hence the constant
$\displaystyle \frac{(N+1)^{2}}{4}$ is also optimal in (\ref{HyUP_curl}) and
is attained by $\displaystyle \mathbf{U}(x)=\alpha e^{-\beta\left\vert
x\right\vert }x$, $\beta>0$. Thus, this new constant is larger than the
constant $\frac{(N-1)^{2}}{4}$ in \eqref{HyUP}.

\section{Main results}

The first principal goal of this article is to study the weighted versions of
the HUP and HyUP for curl-free vector fields which are described in
\eqref{type1}. The second goal is to analyze the weighted second order
inequalities in \eqref{type2}. We note that, since we are dealing with weights,
it is not true in general that for curl-free vector fields $\mathbf{U}$ with a
scalar potential $u$ (i.e. $\mathbf{U}=\nabla u$) there holds
\[
\int_{\mathbb{R}^{N}}\frac{|\nabla\mathbf{U}|^{2}}{|x|^{2a}}\mathrm{dx}%
=\int_{\mathbb{R}^{N}}\frac{|\Delta u|^{2}}{|x|^{2a}}\mathrm{dx},
\]
unless $a=0$. Therefore, the family of inequalities in
\eqref{type1}-\eqref{type2} are independent in general when $a\neq0$. That is,
the spherical harmonics decomposition used in \cite{CFL20} is not applicable
in this situation for the family \eqref{type1} and becomes more difficult to
be applied for the family \eqref{type2}.

Therefore, a new approach needs to be used to establish the weighted cases in
\eqref{type1}. In this paper, we will apply a new representation for curl-free
vector fields developed in \cite{HT21} to establish a very general $L^{2}$-CKN
inequality for curl-free vector fields that contains the HUP and HyUP for
curl-free vector fields as specific cases.
To prove the second order inequalities in \eqref{type2} we apply the expanding
square method in a new fashion way being inspired by our previous work
\cite{CFL21} where we did it for first order inequalities.

In order to state the main results we need to introduce some notations and
definitions regarding the functional framework.
Thus, let $X_{a,b}\left(  \mathbb{R}^{N}\right)  $ be the set of vector fields
$\mathbf{U}\in C^{\infty}\left(  \mathbb{R}^{N}\setminus\left\{  0\right\}
\right)  ^{N}$ such that
\[
\int_{\mathbb{R}^{N}}\frac{|\nabla\mathbf{U}|^{2}}{|x|^{2a}}\mathrm{dx}%
<\infty\text{, }\int_{\mathbb{R}^{N}}\frac{|\mathbf{U}|^{2}}{\left\vert
x\right\vert ^{2b}}\mathrm{dx}<\infty\text{, }\int_{\mathbb{R}^{N}}%
\frac{|\mathbf{U}|^{2}}{\left\vert x\right\vert ^{a+b+1}}\mathrm{dx}<\infty,
\]
\[
\lim_{\left\vert x\right\vert \rightarrow0,\infty}\left\vert x\right\vert
^{-1+\frac{N}{2}-a}\mathbf{U}\left(  x\right)  =0
\]
and
\[
\lim_{\left\vert x\right\vert \rightarrow0,\infty}\left\vert x\right\vert
^{-b+\frac{N}{2}}\mathbf{U}\left(  x\right)  =0.
\]
The first main result of this paper is the following $L^{2}$-CKN inequality
for curl-free vector fields:

\begin{theorem}
\label{T1}Let $a,b$ be real numbers such that $\displaystyle\left(  \frac
{N}{2}-a\right)  ^{2}\geq N+1$. We have for $\mathbf{U}\in X_{a,b}\left(
\mathbb{R}^{N}\right)  $ with $\operatorname{curl}\mathbf{U}=0$ that

(1) if $a-b+1>0$, then
\[
\int_{\mathbb{R}^{N}}\frac{|\nabla\mathbf{U}|^{2}}{|x|^{2a}}\mathrm{dx}%
\int_{\mathbb{R}^{N}}\frac{|\mathbf{U}|^{2}}{\left\vert x\right\vert ^{2b}%
}\mathrm{dx}\geq C^{2}(N,a,b) \left(  \int_{\mathbb{R}^{N}}\frac
{|\mathbf{U}|^{2}}{\left\vert x\right\vert ^{a+b+1}}\mathrm{dx}\right)  ^{2},
\]
where
\[
C(N,a,b) = \sqrt{\left(  1-\frac{N}{2}+a\right)  ^{2}+N-1}+\frac{a-b+1}{2}%
\]
is sharp and can be attained by the curl-free vector fields
\[
\displaystyle\mathbf{U}\left(  x\right)  =\gamma\left\vert x\right\vert
^{-\frac{N}{2}+a+\sqrt{\left(  1-\frac{N}{2}+a\right)  ^{2}+N-1}}%
e^{-\frac{\beta}{\left(  a-b+1\right)  }\left\vert x\right\vert ^{\left(
a-b+1\right)  }}x,\quad\gamma\in\mathbb{R}, \quad\beta>0;
\]

(2) if $a-b+1<0$, then%
\[
\int_{\mathbb{R}^{N}}\frac{|\nabla\mathbf{U}|^{2}}{|x|^{2a}}\mathrm{dx}%
\int_{\mathbb{R}^{N}}\frac{|\mathbf{U}|^{2}}{\left\vert x\right\vert ^{2b}%
}\mathrm{dx}\geq C^{2}(N,a,b)\left(  \int_{\mathbb{R}^{N}}\frac{|\mathbf{U}%
|^{2}}{\left\vert x\right\vert ^{a+b+1}}\mathrm{dx}\right)  ^{2},
\]
where
\[
C(N,a,b) = \sqrt{\left(  1-\frac{N}{2}+a\right)  ^{2}+N-1}-\frac{a-b+1}{2}%
\]
is sharp and can be attained by the curl-free vector fields
\[
\displaystyle\mathbf{U}\left(  x\right)  =\gamma\left\vert x\right\vert
^{-\frac{N}{2}+a-\sqrt{\left(  1-\frac{N}{2}+a\right)  ^{2}+N-1}}%
e^{-\frac{\beta}{\left(  a-b+1\right)  }\left\vert x\right\vert ^{\left(
a-b+1\right)  }}x, \quad\gamma\in\mathbb{R}, \quad\beta<0.
\]

\end{theorem}

Here are some direct consequences of our Theorem \ref{T1}. Let $b=-a-1$, we
get the following weighted HUP for curl-free vector fields:

\begin{corollary}
\label{T1.1}Let $a$ be such that $a>-1$ and $\displaystyle\left(  \frac{N}%
{2}-a\right)  ^{2}\geq N+1$. Then we have for $\mathbf{U}\in X_{a,-a-1}\left(
\mathbb{R}^{N}\right)  $ with $\operatorname{curl}\mathbf{U}=0$ that
\[
\int_{\mathbb{R}^{N}}\frac{|\nabla\mathbf{U}|^{2}}{|x|^{2a}}\mathrm{dx}%
\int_{\mathbb{R}^{N}}|x|^{2a+2}|\mathbf{U}|^{2}\mathrm{dx}\geq C^{2}(N,a,-a-1)
\left(  \int_{\mathbb{R}^{N}}|\mathbf{U}|^{2}\mathrm{dx}\right)  ^{2},
\]
where
\[
C(N,a,-a-1) = \sqrt{\left(  1-\frac{N}{2}+a\right)  ^{2}+N-1}+1+a
\]
is sharp and can be attained by curl-free vector fields
\[
\displaystyle\mathbf{U}\left(  x\right)  =\gamma\left\vert x\right\vert
^{-\frac{N}{2}+a+\sqrt{\left(  1-\frac{N}{2}+a\right)  ^{2}+N-1}}%
e^{-\frac{\beta}{2\left(  a+1\right)  }\left\vert x\right\vert ^{2\left(
a+1\right)  }}x,\quad\gamma\in\mathbb{R}, \quad\beta>0.
\]

\end{corollary}

When $b=-a$, we obtain the weighted HyUP for curl-free vector fields:

\begin{corollary}
\label{T1.2}Let $a$ be such that $a>-\frac{1}{2}$ and $\displaystyle\left(
\frac{N}{2}-a\right)  ^{2}\geq N+1$. Then we have for $\mathbf{U}\in
X_{a,-a}\left(  \mathbb{R}^{N}\right)  $ with $\operatorname{curl}%
\mathbf{U}=0$ that
\[
\int_{\mathbb{R}^{N}}\frac{|\nabla\mathbf{U}|^{2}}{|x|^{2a}}\mathrm{dx}%
\int_{\mathbb{R}^{N}}|x|^{2a}|\mathbf{U}|^{2}\mathrm{dx}\geq C^{2}(N,a,-a)
\left(  \int_{\mathbb{R}^{N}}\frac{\left\vert \mathbf{U}\right\vert ^{2}}%
{|x|}\mathrm{dx}\right)  ^{2}%
\]
where
\[
C(N,a,-a) = \sqrt{\left(  1-\frac{N}{2}+a\right)  ^{2}+N-1}+\frac{1}{2}+a
\]
is sharp and can be attained by curl-free vector fields
\[
\displaystyle\mathbf{U}\left(  x\right)  =\gamma\left\vert x\right\vert
^{-\frac{N}{2}+a+\sqrt{\left(  1-\frac{N}{2}+a\right)  ^{2}+N-1}}%
e^{-\frac{\beta}{2a+1}\left\vert x\right\vert ^{2a+1}}x,\quad\gamma
\in\mathbb{R},\quad\beta>0.
\]

\end{corollary}

In particular, when $a=0$, we recover the HUP (\ref{HUP_curl}) and HyUP
(\ref{HyUP_curl}) for curl-free vector fields.

Obviously, using $\mathbf{U}=\nabla u$, Theorem \ref{T1} yields

\begin{corollary}
\label{T1.3}Let $a,b$ be real numbers such that $\displaystyle\left(  \frac
{N}{2}-a\right)  ^{2}\geq N+1$. We have for $u$ such that $\nabla u\in
X_{a,b}\left(  \mathbb{R}^{N}\right)  $ that

(1) if $a-b+1>0$, then
\[
\int_{\mathbb{R}^{N}}\frac{|D^{2}u|^{2}}{|x|^{2a}}\mathrm{dx}\int
_{\mathbb{R}^{N}}\frac{|\nabla u|^{2}}{\left\vert x\right\vert ^{2b}%
}\mathrm{dx}\geq C^{2}(N,a,b) \left(  \int_{\mathbb{R}^{N}}\frac{|\nabla
u|^{2}}{\left\vert x\right\vert ^{a+b+1}}\mathrm{dx}\right)  ^{2},
\]
where
\[
C(N,a,b) = \sqrt{\left(  1-\frac{N}{2}+a\right)  ^{2}+N-1}+\frac{a-b+1}{2}%
\]
is sharp and can be attained by the functions $u$ such that
\[
\displaystyle\nabla u= \gamma\left\vert x\right\vert ^{-\frac{N}{2}%
+a+\sqrt{\left(  1-\frac{N}{2}+a\right)  ^{2}+N-1}}e^{-\frac{\beta}{\left(
a-b+1\right)  }\left\vert x\right\vert ^{\left(  a-b+1\right)  }}x,\quad
\gamma\in\mathbb{R}, \quad\beta>0;
\]

(2) if $a-b+1<0$, then
\[
\int_{\mathbb{R}^{N}}\frac{|D^{2}u|^{2}}{|x|^{2a}}\mathrm{dx}\int
_{\mathbb{R}^{N}}\frac{|\nabla u|^{2}}{\left\vert x\right\vert ^{2b}%
}\mathrm{dx}\geq C^{2}(N,a,b) \left(  \int_{\mathbb{R}^{N}}\frac{|\nabla
u|^{2}}{\left\vert x\right\vert ^{a+b+1}}\mathrm{dx}\right)  ^{2},
\]
where
\[
C(N,a,b) = \sqrt{\left(  1-\frac{N}{2}+a\right)  ^{2}+N-1}-\frac{a-b+1}{2}%
\]
is sharp and can be attained by the functions $u$ such that
\[
\displaystyle\nabla u=\gamma\left\vert x\right\vert ^{-\frac{N}{2}%
+a-\sqrt{\left(  1-\frac{N}{2}+a\right)  ^{2}+N-1}}e^{-\frac{\beta}{\left(
a-b+1\right)  }\left\vert x\right\vert ^{\left(  a-b+1\right)  }}x,\quad
\gamma\in\mathbb{R},\quad\beta<0.
\]

\end{corollary}

In particular, when $a=0$, we have that

\begin{corollary}
\label{T1.4}Let $N\geq5$. We have for $u$ such that $\nabla u\in
X_{0,b}\left(  \mathbb{R}^{N}\right)  $ that

(1) if $b<1$, then
\[
\int_{\mathbb{R}^{N}}|\Delta u|^{2}\mathrm{dx}\int_{\mathbb{R}^{N}}%
\frac{|\nabla u|^{2}}{\left\vert x\right\vert ^{2b}}\mathrm{dx}\geq\left(
\frac{N-b+1}{2}\right)  ^{2}\left(  \int_{\mathbb{R}^{N}}\frac{|\nabla u|^{2}%
}{\left\vert x\right\vert ^{b+1}}\mathrm{dx}\right)  ^{2},
\]
where the constant $\displaystyle\left(  \frac{N-b+1}{2}\right)  ^{2}$ is
sharp and can be attained by the functions $u$ such that $\displaystyle\nabla
u=$ $\gamma e^{-\frac{\beta}{1-b}\left\vert x\right\vert ^{1-b}}x$, $\gamma\in%
%TCIMACRO{\U{211d} }%
%BeginExpansion
\mathbb{R}
%EndExpansion
,$ $\beta>0$;

(2) if $b>1$, then%
\[
\int_{\mathbb{R}^{N}}|\Delta u|^{2}\mathrm{dx}\int_{\mathbb{R}^{N}}%
\frac{|\nabla u|^{2}}{\left\vert x\right\vert ^{2b}}\mathrm{dx}\geq\left(
\frac{N+b-1}{2}\right)  ^{2}\left(  \int_{\mathbb{R}^{N}}\frac{|\nabla u|^{2}%
}{\left\vert x\right\vert ^{b+1}}\mathrm{dx}\right)  ^{2},
\]
where the constant $\displaystyle\left(  \frac{N+b-1}{2}\right)  ^{2}$ is
sharp and can be attained by the functions $u$ such that $\displaystyle\nabla
u=\gamma\left\vert x\right\vert ^{-N}e^{-\frac{\beta}{1-b}\left\vert
x\right\vert ^{1-b}}x$, $\gamma\in%
%TCIMACRO{\U{211d} }%
%BeginExpansion
\mathbb{R}
%EndExpansion
,$ $\beta<0$.
\end{corollary}

As mentioned earlier, it is not true in general that for $\mathbf{U}=\nabla
u$
\[
\int_{\mathbb{R}^{N}}\frac{|\nabla\mathbf{U}|^{2}}{|x|^{2a}}\mathrm{dx}%
=\int_{\mathbb{R}^{N}}\frac{|\Delta u|^{2}}{|x|^{2a}}\mathrm{dx}.
\]
Therefore, Theorem \ref{T1} does not imply weighted versions of the second
order HUP \eqref{HUP_curl_scalar} or HyUP \eqref{HyUP_curl_scalar}. That is,
Theorem \ref{T1} does not imply that%
\[
\int_{\mathbb{R}^{N}}\frac{|\Delta u|^{2}}{|x|^{2a}}\mathrm{dx}\int
_{\mathbb{R}^{N}}|x|^{2a+2}\left\vert \nabla u\right\vert ^{2}\mathrm{dx}\geq
C\left(  N,a\right)  \left(  \int_{\mathbb{R}^{N}}|\nabla u|^{2}%
\mathrm{dx}\right)  ^{2}.
\]
Motivated by this observation, the second goal of our paper is to study the
weighted second order inequalities in \eqref{type2}. First, let $Y_{a}\left(
\mathbb{R}^{N}\right)  $ be the set of scalar-valued functions $u\in
C^{\infty}\left(  \mathbb{R}^{N}\setminus\left\{  0\right\}  \right)  $ such
that
\begin{equation}
\int_{\mathbb{R}^{N}}\frac{|\Delta u|^{2}}{|x|^{2a}}\mathrm{dx}<\infty
,\int_{\mathbb{R}^{N}}|x|^{2a+2}\left\vert \frac{x}{\left\vert x\right\vert
}\cdot\nabla u\right\vert ^{2}\mathrm{dx}<\infty,\int_{\mathbb{R}^{N}}|\nabla
u|^{2}\mathrm{dx}<\infty,
\end{equation}%
\[
\lim_{\left\vert x\right\vert \rightarrow0,\infty}\left\vert x\right\vert
^{N-1}\left\vert u\left(  x\right)  \right\vert ^{2}=0,
\]%
\[
\lim_{\left\vert x\right\vert \rightarrow0,\infty}\left\vert x\right\vert
^{N}\left\vert \frac{x}{|x|}\cdot\nabla u\left(  x\right)  \right\vert
^{2}=0.
\]
and
\[
\lim_{\left\vert x\right\vert \rightarrow0,\infty}\left\vert x\right\vert
^{2a+N}\left\vert u\left(  x\right)  \right\vert ^{2}=0.
\]
Then we will prove that

\begin{theorem}
\label{T2} Let $N\geq1$ and $a\in\mathbb{R}$. For all $u\in Y_{a}\left(
\mathbb{R}^{N}\right)  $, there holds
\begin{equation}
\int_{\mathbb{R}^{N}}\frac{|\Delta u|^{2}}{|x|^{2a}}\mathrm{dx}\int
_{\mathbb{R}^{N}}|x|^{2a+2}\left\vert \frac{x}{\left\vert x\right\vert }%
\cdot\nabla u\right\vert ^{2}\mathrm{dx}\geq\frac{(N+2+4a)^{2}}{4}\left(
\int_{\mathbb{R}^{N}}|\nabla u|^{2}\mathrm{dx}\right)  ^{2}.\quad
\label{wHUP_curl_scalar}%
\end{equation}
If either $\min\left\{  a+1,N+2+4a\right\}  >0$ or $\max\left\{
a+1,N+2+4a\right\}  <0$, then the constant $\displaystyle\frac{(N+2+4a)^{2}%
}{4}$ is optimal in (\ref{wHUP_curl_scalar}) and is attained by functions of
the form
\begin{equation}
u(x)=\gamma e^{-\beta|x|^{2(1+a)}},\text{ }\gamma\in%
%TCIMACRO{\U{211d} }%
%BeginExpansion
\mathbb{R}
%EndExpansion
,\text{ }\beta>0.\label{minimizers}%
\end{equation}
If $N\geq2$ and either $a+1>0$ or $N+2a<0$, the equality also happens in
(\ref{wHUP_curl_scalar}) for infinitely many nonradial functions of the
form%
\begin{align*}
u(x) &  =|x|^{\alpha}{}_{1}F_{1}\left(  \frac{\alpha+N+2a}{2a+2};\frac
{2\alpha+2a+N}{2a+2};-\frac{t}{2a+2}|x|^{2a+2}\right)  g\left(\frac{x}{|x|}\right),\\
\alpha &  =\frac{2-N+\operatorname{sgn}(a+1)\sqrt{(N-2)^{2}-4\lambda}}{2}.
\end{align*}
with $t\in\mathbb{R}$ such that $\frac{t}{2a+2}>0$, $\Delta_{\sigma}g=\lambda
g$ for some $\lambda=-k\left(  N+k-2\right)  $, $k=1,2,...$, and $_{1}%
F_{1}(A;B;z)$ are the Kummer's confluent hypergeometric functions, that are
the solutions to Kummer's equation
\begin{equation}
z\frac{d^{2}w}{dz^{2}}+(B-z)\frac{dw}{dz}-Aw=0.
\end{equation}

\end{theorem}

In the particular case $a=0$, we get
\begin{equation}
\int_{\mathbb{R}^{N}}|\Delta u|^{2}\mathrm{dx}\int_{\mathbb{R}^{N}}%
|x|^{2}\left\vert \frac{x}{\left\vert x\right\vert }\cdot\nabla u\right\vert
^{2}\mathrm{dx}\geq\frac{(N+2)^{2}}{4}\left(  \int_{\mathbb{R}^{N}}|\nabla
u|^{2}\mathrm{dx}\right)  ^{2},\quad\forall u\in C_{0}^{\infty}(\mathbb{R}%
^{N}),
\end{equation}
which implies the second order HUP (\ref{HUP_div_scalar}) since $\left\vert
\frac{x}{\left\vert x\right\vert }\cdot\nabla u\right\vert \leq\left\vert
\nabla u\right\vert $.

We note that in the process of preparing our manuscript, we have noticed that
(\ref{wHUP_curl_scalar}) has been investigated in \cite{DN21}. Nevertheless,
our approach in this paper is different and much simpler than the one in
\cite{DN21}. Our method also provides nonradial optimizers of
(\ref{wHUP_curl_scalar}).

As a consequence of Theorem \ref{T2}, we obtain

\begin{corollary}
\label{T2.1}Let $N\geq2$ and $a\in\mathbb{R}$. We have for $u\in Y_{a}\left(
\mathbb{R}^{N}\right)  $ that%
\begin{equation}
\int_{\mathbb{R}^{N}}\frac{|\Delta u|^{2}}{|x|^{2a}}\mathrm{dx}\int
_{\mathbb{R}^{N}}|x|^{2a+2}\left\vert \nabla u\right\vert ^{2}\mathrm{dx}%
\geq\frac{(N+2+4a)^{2}}{4}\left(  \int_{\mathbb{R}^{N}}|\nabla u|^{2}%
\mathrm{dx}\right)  ^{2}.\label{wHUP}%
\end{equation}
If either $a+1>0$ or $N+2+4a<0$, then the constant $\frac{(N+2+4a)^{2}}{4}$ is
optimal and can be attained by the optimizers of the form
\[
u(x)=\gamma e^{-\beta|x|^{2(1+a)}},\text{ }\gamma\in%
%TCIMACRO{\U{211d} }%
%BeginExpansion
\mathbb{R}
%EndExpansion
,\text{ }\beta>0.
\]

\end{corollary}

\section{Preliminary}

For the sake of clarity let us recall few important aspects about the curl
operator. While the divergence operator is clearly defined in any dimension
$N\geq2$ through the formula $\text{div}\mathbf{U}:=\sum_{j=1}^{N}%
\frac{\partial U_{j}}{x_{j}}$, the curl operator can be directly defined only
in dimensions $2$ and $3$; for higher dimensions it is well-known that it is understood
in any dimension via differential forms. The $\operatorname{curl}$ of a vector
field $\mathbf{U}=\left(  U_{1},...,U_{N}\right)  \in\left(  C^{\infty}\left(
\mathbb{R}^{N}\right)  \right)  ^{N}$ is defined as the differential $2$-form%
\[
\operatorname{curl}\mathbf{U}=d\left(  \mathbf{U}\cdot dx\right)  =d\left(
%TCIMACRO{\dsum \limits_{j=1}^{N}}%
%BeginExpansion
{\displaystyle\sum\limits_{j=1}^{N}}
%EndExpansion
U_{j}dx_{j}\right)  \text{,}%
\]
where $d$ denotes the exterior differential. In the standard Euclidean
coordinates, we can write%
\[
d\left(  \mathbf{U}\cdot dx\right)  =%
%TCIMACRO{\dsum \limits_{j=1}^{N}}%
%BeginExpansion
{\displaystyle\sum\limits_{j=1}^{N}}
%EndExpansion
dU_{j}\wedge dx_{j}=\underset{j<k}{%
%TCIMACRO{\dsum }%
%BeginExpansion
{\displaystyle\sum}
%EndExpansion%
%TCIMACRO{\dsum }%
%BeginExpansion
{\displaystyle\sum}
%EndExpansion
}\left(  \frac{\partial U_{k}}{dx_{j}}-\frac{\partial U_{j}}{dx_{k}}\right)
dx_{j}\wedge dx_{k}.
\]
Therefore, we have that $\operatorname{curl}\mathbf{U}=0$ if and only if for
all $1\leq j,k\leq N:$
\[
\frac{\partial U_{k}}{dx_{j}}=\frac{\partial U_{j}}{dx_{k}}.
\]
This also implies that any curl-free vector field $\mathbf{U}$ has a scalar
potential $u\in C^{\infty}\left(  \mathbb{R}^{N}\right)$ satisfying $ \mathbf{U}=\nabla
u$. Indeed, we can always choose $u\left(  x\right)  =%
%TCIMACRO{\dint \limits_{0}^{\left\vert x\right\vert }}%
%BeginExpansion
{\displaystyle\int\limits_{0}^{\left\vert x\right\vert }}
%EndExpansion
\frac{x}{\left\vert x\right\vert }\cdot\mathbf{U}\left(  \rho\frac
{x}{\left\vert x\right\vert }\right)  d\rho$.

Let $\mathbf{U}=\left(  U_{1},...,U_{N}\right)  $ be a vector field. Then we
can write
\[
\mathbf{U}=\sigma U_{R}+\mathbf{U}_{S}%
\]
for all $x=r\sigma$ where $r=\left\vert x\right\vert $, $\sigma=\frac
{x}{\left\vert x\right\vert }$, $U_{R}=U_{R}\left(  x\right)  $ is the radial
scalar component and $\mathbf{U}_{S}=\mathbf{U}_{S}\left(  x\right)  $ is the
spherical vector part. In particular, $\sigma\cdot\mathbf{U}_{S}=0$.

We also denote $\partial_{r}\varphi=\sigma\cdot\nabla\varphi$ and
$\nabla_{\sigma}\varphi=r\left(  \nabla\varphi\right)  _{S}$ and so
\[
\nabla=\sigma\partial_{r}+\frac{1}{r}\nabla_{\sigma}\text{.}%
\]
Then it is known that%
\begin{equation}
\Delta=\partial_{rr}+\frac{N-1}{r}\partial_{r}+\frac{1}{r^{2}}\Delta_{\sigma
}\text{,}%
  \label{eq:spherical-decomp-laplacian}
\end{equation}
where $\Delta_{\sigma}$ is the Laplace-Beltrami operator on the sphere.

In \cite{HT21}, the following representation of curl-free fields has been
established. Let $\lambda\in%
%TCIMACRO{\U{211d} }%
%BeginExpansion
\mathbb{R}
%EndExpansion
$ and let $\mathbf{V}=r^{1-\lambda}\mathbf{U}$. Then $\mathbf{U}\in
C_{0}^{\infty}\left(  \mathbb{R}^{N}\right)  ^{N}$ is curl-free if and only if
there exist two scalar fields $f,\varphi\in C^{\infty}\left(  \mathbb{R}%
^{N}\setminus\left\{  0\right\}  \right)  $ satisfying $f$ is radially
symmetric,
\[
\int_{\mathbb{S}^{N-1}}\varphi\left(  r\sigma\right)  d\sigma=0\text{ for all
}r>0\text{,}%
\]
and
\[
\mathbf{V}=\sigma\left(  f+\left(  \lambda+\partial_{t}\right)  \varphi
\right)  +\nabla_{\sigma}\varphi\text{ on }\mathbb{R}^{N}\setminus\left\{
0\right\}  \text{.}%
\]
Moreover, if $\mathbf{U}$ has a compact support on $\mathbb{R}^{N}%
\setminus\left\{  0\right\}  $, then so do $f$ and $\varphi$.

Now, we proceed as in \cite{HT21} and let $\mathbf{V}=r^{1-\lambda}\mathbf{U}$
with $\lambda=2-\frac{N}{2}+a$. Then we have with $t=\ln r$ that%
\[
\left\vert x\right\vert ^{\gamma}dx=r^{\gamma+N-1}drd\sigma=r^{\gamma+N}%
\frac{dr}{r}d\sigma=r^{\gamma+N}dtd\sigma\text{.}%
\]
Therefore
\begin{align*}
\int_{\mathbb{R}^{N}}\frac{|\mathbf{U}|^{2}}{\left\vert x\right\vert ^{2b}%
}\mathrm{dx}  &  =\int_{\mathbb{R}^{N}}\left\vert x\right\vert ^{-2b}%
\left\vert x\right\vert ^{2\lambda-2}|\mathbf{V}|^{2}\mathrm{dx}\\
&  =\int_{\mathbb{R}_{+}\times\mathbb{S}^{N-1}}r^{2\lambda-2-2b+N-1}%
|\mathbf{V}|^{2}\mathrm{drd\sigma}\\
&  =\int_{\mathbb{R}\times\mathbb{S}^{N-1}}e^{\left(  2a-2b+2\right)
t}|\mathbf{V}|^{2}\mathrm{dtd\sigma}%
\end{align*}
and
\begin{align*}
\int_{\mathbb{R}^{N}}\frac{|\mathbf{U}|^{2}}{\left\vert x\right\vert ^{a+b+1}%
}\mathrm{dx}  &  =\int_{\mathbb{R}^{N}}\left\vert x\right\vert ^{-a-b-1}%
\left\vert x\right\vert ^{2\lambda-2}|\mathbf{V}|^{2}\mathrm{dx}\\
&  =\int_{\mathbb{R}_{+}\times\mathbb{S}^{N-1}}r^{2\lambda-2-a-b-1+N-1}%
|\mathbf{V}|^{2}\mathrm{drd\sigma}\\
&  =\int_{\mathbb{R}\times\mathbb{S}^{N-1}}e^{\left(  a-b+1\right)
t}|\mathbf{V}|^{2}\mathrm{dtd\sigma.}%
\end{align*}
Also,
\begin{align*}
\int_{\mathbb{R}^{N}}\frac{\left\vert \nabla\mathbf{U}\right\vert ^{2}%
}{|x|^{2a}}\mathrm{dx}  &  =\int_{\mathbb{R}_{+}\times\mathbb{S}^{N-1}}\left(
\left\vert \partial_{r}\mathbf{U}\right\vert ^{2}+\frac{1}{r^{2}}\left\vert
\nabla_{\sigma}\mathbf{U}\right\vert ^{2}\right)  r^{N-1-2a}\mathrm{drd\sigma
}\\
&  =\int_{\mathbb{R}_{+}\times\mathbb{S}^{N-1}}\left(  \left\vert \partial
_{r}\left(  r^{\lambda-1}\mathbf{V}\right)  \right\vert ^{2}+\frac{1}{r^{2}%
}\left\vert \nabla_{\sigma}\left(  r^{\lambda-1}\mathbf{V}\right)  \right\vert
^{2}\right)  r^{3-2\lambda}\mathrm{drd\sigma}\\
&  =\int_{\mathbb{R}\times\mathbb{S}^{N-1}}\left(  \left\vert \left(
\lambda-1\right)  \mathbf{V}+\partial_{t}\mathbf{V}\right\vert ^{2}+\left\vert
\nabla_{\sigma}\mathbf{V}\right\vert ^{2}\right)  \mathrm{dtd\sigma}\\
&  =\int_{\mathbb{R}\times\mathbb{S}^{N-1}}\left(  \left(  \lambda-1\right)
^{2}\left\vert \mathbf{V}\right\vert ^{2}+\left\vert \partial_{t}%
\mathbf{V}\right\vert ^{2}+\left\vert \nabla_{\sigma}\mathbf{V}\right\vert
^{2}\right)  \mathrm{dtd\sigma}\text{.}%
\end{align*}
Now, using $\mathbf{V}=\sigma\left(  f+\left(  \lambda+\partial_{t}\right)
\varphi\right)  +\nabla_{\sigma}\varphi$, we have
\begin{align*}
\Delta_{\sigma}\mathbf{V}  &  \mathbf{=}\Delta_{\sigma}\left(  \sigma\left(
f+\left(  \lambda+\partial_{t}\right)  \varphi\right)  +\nabla_{\sigma}%
\varphi\right) \\
&  =\left(  \sigma\Delta_{\sigma}+2\nabla_{\sigma}-\left(  N-1\right)
\sigma\right)  \left(  f+\left(  \lambda+\partial_{t}\right)  \varphi\right)
\\
&  +\left(  \nabla_{\sigma}\Delta_{\sigma}+\left(  N-3\right)  \nabla_{\sigma
}-2\sigma\Delta_{\sigma}\right)  \varphi\\
&  =\sigma\left(  \left(  \partial_{t}+\lambda-2\right)  \Delta_{\sigma
}\varphi\right)  -\left(  N-1\right)  \sigma\left(  f+\left(  \lambda
+\partial_{t}\right)  \varphi\right) \\
&  +\nabla_{\sigma}\left(  2\partial_{t}+\Delta_{\sigma}+2\lambda+N-3\right)
\varphi\\
&  =\sigma\left(  \left(  \partial_{t}+\lambda-2\right)  \Delta_{\sigma
}\varphi\right)  +\nabla_{\sigma}\left(  2\partial_{t}+\Delta_{\sigma
}+2\lambda+2N-4\right)  \varphi\\
&  -\left(  N-1\right)  \mathbf{V}\text{.}%
\end{align*}
Therefore we get by integration by parts that%
\begin{align*}
&  \int_{\mathbb{R}\times\mathbb{S}^{N-1}}\left\vert \nabla_{\sigma}%
\mathbf{V}\right\vert ^{2}\mathrm{dtd\sigma}\\
&  =-\int_{\mathbb{R}\times\mathbb{S}^{N-1}}\mathbf{V}\cdot\left(
\Delta_{\sigma}\mathbf{V}\right)  \mathrm{dtd\sigma}\\
&  =-\int_{\mathbb{R}\times\mathbb{S}^{N-1}}\left(  f+\left(  \lambda
+\partial_{t}\right)  \varphi\right)  \left(  \partial_{t}+\lambda-2\right)
\Delta_{\sigma}\varphi\mathrm{dtd\sigma}\\
&  +\int_{\mathbb{R}\times\mathbb{S}^{N-1}}\left(  -\nabla_{\sigma}%
\varphi\cdot\nabla_{\sigma}\left(  2\partial_{t}+\Delta_{\sigma}%
+2\lambda+2N-4\right)  \varphi+\left(  N-1\right)  \left\vert \mathbf{V}%
\right\vert ^{2}\right)  \mathrm{dtd\sigma}\\
&  =\int_{\mathbb{R}\times\mathbb{S}^{N-1}}\left(  \left(  \Delta_{\sigma
}\varphi\right)  ^{2}+\left(  \lambda^{2}-4\lambda-2N+4\right)  \left\vert
\nabla_{\sigma}\varphi\right\vert ^{2}\right)  \mathrm{dtd\sigma}\\
&  +\int_{\mathbb{R}\times\mathbb{S}^{N-1}}\left(  \left\vert \partial
_{t}\nabla_{\sigma}\varphi\right\vert ^{2}+\left(  N-1\right)  \left\vert
\mathbf{V}\right\vert ^{2}\right)  \mathrm{dtd\sigma}\text{.}%
\end{align*}
Using the fact that the spectrum of \ $-\Delta_{\sigma}$ is the set $\left\{
\kappa\left(  N+\kappa-2\right)  \text{, }\kappa=0,1,...\right\}  $, we obtain
that for all $\varphi$ such that $\int_{\mathbb{S}^{N-1}}\varphi
\mathrm{d\sigma}=0:$
\begin{align*}
&  \int_{\mathbb{R}\times\mathbb{S}^{N-1}}\left(  \left(  \Delta_{\sigma
}\varphi\right)  ^{2}+\left(  \lambda^{2}-4\lambda-2N+4\right)  \left\vert
\nabla_{\sigma}\varphi\right\vert ^{2}\right)  \mathrm{dtd\sigma}\\
&  \geq\min_{\kappa}\left\{  \kappa^{2}\left(  N+\kappa-2\right)  ^{2}+\left(
\lambda^{2}-4\lambda-2N+4\right)  \kappa\left(  N+\kappa-2\right)  \right\}
\int_{\mathbb{R}\times\mathbb{S}^{N-1}}\left\vert \varphi\right\vert
^{2}\mathrm{dtd\sigma}\\
&  =\left(  N-1\right)  \left(  \left(  \lambda-2\right)  ^{2}-N-1\right)
\int_{\mathbb{R}\times\mathbb{S}^{N-1}}\left\vert \varphi\right\vert
^{2}\mathrm{dtd\sigma}.
\end{align*}
Also, since $\int_{\mathbb{S}^{N-1}}\varphi\mathrm{d\sigma}=0$, we get
$\int_{\mathbb{S}^{N-1}}\partial_{t}\varphi\mathrm{d\sigma}=0$ and therefore
\[
\int_{\mathbb{S}^{N-1}}\left\vert \nabla_{\sigma}\left(  \partial_{t}%
\varphi\right)  \right\vert ^{2}\mathrm{d\sigma}\geq\left(  N-1\right)
\int_{\mathbb{S}^{N-1}}\left\vert \partial_{t}\varphi\right\vert
^{2}\mathrm{d\sigma}.
\]
Thus, we get%
\begin{align*}
&  \int_{\mathbb{R}\times\mathbb{S}^{N-1}}\left\vert \nabla_{\sigma}%
\mathbf{V}\right\vert ^{2}\mathrm{dtd\sigma}\\
&  \geq\left(  N-1\right)  \int_{\mathbb{R}\times\mathbb{S}^{N-1}}\left(
\left\vert \mathbf{V}\right\vert ^{2}+\left(  \left(  \lambda-2\right)
^{2}-N-1\right)  \left\vert \varphi\right\vert ^{2}+\left\vert \partial
_{t}\varphi\right\vert ^{2}\right)  \mathrm{dtd\sigma}.
\end{align*}
Hence%
\begin{align*}
&  \int_{\mathbb{R}^{N}}\frac{\left\vert \nabla\mathbf{U}\right\vert ^{2}%
}{|x|^{2a}}\mathrm{dx}\\
&  \geq\left(  \left(  \lambda-1\right)  ^{2}+N-1\right)  \int_{\mathbb{R}%
\times\mathbb{S}^{N-1}}\left\vert \mathbf{V}\right\vert ^{2}\mathrm{dtd\sigma
}+\int_{\mathbb{R}\times\mathbb{S}^{N-1}}\left\vert \partial_{t}%
\mathbf{V}\right\vert ^{2}\mathrm{dtd\sigma}\\
&  +\left(  N-1\right)  \int_{\mathbb{R}\times\mathbb{S}^{N-1}}\left(  \left(
\left(  \lambda-2\right)  ^{2}-N-1\right)  \left\vert \varphi\right\vert
^{2}+\left\vert \partial_{t}\varphi\right\vert ^{2}\right)  \mathrm{dtd\sigma
}.
\end{align*}

\section{Weighted $L^{2}$-Caffarelli-Kohn-Nirenberg inequalities for curl-free
vector fields-Proof of Theorem \ref{T1}}

\begin{proof}
[Proof of Theorem \ref{T1}]Let $\mathbf{U}\in X_{a,b}\left(  \mathbb{R}%
^{N}\right)  $. Then we can find two scalar fields $f,\varphi\in C^{\infty
}\left(  \mathbb{R}^{N}\setminus\left\{  0\right\}  \right)  $ satisfying $f$
is radially symmetric,
\[
\int_{\mathbb{S}^{N-1}}\varphi\left(  r\sigma\right)  d\sigma=0\text{ for all
}r>0\text{,}%
\]
and
\[
\mathbf{V}=\sigma\left(  f+\left(  \lambda+\partial_{t}\right)  \varphi
\right)  +\nabla_{\sigma}\varphi\text{ on }\mathbb{R}^{N}\setminus\left\{
0\right\}  \text{,}%
\]
where $\mathbf{V}=r^{1-\lambda}\mathbf{U}$ with $\lambda=2-\frac{N}{2}+a$. Let
$t=\ln r$. We have
\[
\left\vert x\right\vert ^{\gamma}dx=r^{\gamma+N-1}drd\sigma=r^{\gamma+N}%
\frac{dr}{r}d\sigma=r^{\gamma+N}dtd\sigma\text{.}%
\]
As in the previous section, we get
\begin{align*}
\int_{\mathbb{R}^{N}}\frac{\left\vert \nabla\mathbf{U}\right\vert ^{2}%
}{|x|^{2a}}\mathrm{dx}  &  \geq\left(  \left(  \lambda-1\right)
^{2}+N-1\right)  \int_{\mathbb{R}\times\mathbb{S}^{N-1}}\left\vert
\mathbf{V}\right\vert ^{2}\mathrm{dtd\sigma}+\int_{\mathbb{R}\times
\mathbb{S}^{N-1}}\left\vert \partial_{t}\mathbf{V}\right\vert ^{2}%
\mathrm{dtd\sigma}\\
&  +\left(  N-1\right)  \int_{\mathbb{R}\times\mathbb{S}^{N-1}}\left(  \left(
\left(  \lambda-2\right)  ^{2}-N-1\right)  \left\vert \varphi\right\vert
^{2}+\left\vert \partial_{t}\varphi\right\vert ^{2}\right)  \mathrm{dtd\sigma
}\\
&  \geq\left(  \left(  \lambda-1\right)  ^{2}+N-1\right)  \int_{\mathbb{R}%
\times\mathbb{S}^{N-1}}\left\vert \mathbf{V}\right\vert ^{2}\mathrm{dtd\sigma
}+\int_{\mathbb{R}\times\mathbb{S}^{N-1}}\left\vert \partial_{t}%
\mathbf{V}\right\vert ^{2}\mathrm{dtd\sigma}%
\end{align*}
if $\left(  \frac{N}{2}-a\right)  ^{2}\geq N+1$.

Also
\[
\int_{\mathbb{R}^{N}}\frac{|\mathbf{U}|^{2}}{\left\vert x\right\vert ^{2b}%
}\mathrm{dx}=\int_{\mathbb{R}\times\mathbb{S}^{N-1}}e^{\left(  2a-2b+2\right)
t}|\mathbf{V}|^{2}\mathrm{dtd\sigma}%
\]
and
\[
\int_{\mathbb{R}^{N}}\frac{|\mathbf{U}|^{2}}{\left\vert x\right\vert ^{a+b+1}%
}\mathrm{dx}=\int_{\mathbb{R}\times\mathbb{S}^{N-1}}e^{\left(  a-b+1\right)
t}|\mathbf{V}|^{2}\mathrm{dtd\sigma.}%
\]
\textbf{Case 1:} $a-b+1>0$

We aim to show that%
\begin{align*}
&  \left[  \int_{\mathbb{R}\times\mathbb{S}^{N-1}}\left\vert \partial
_{t}\mathbf{V}\right\vert ^{2}\mathrm{dtd\sigma}+\left(  \left(  1-\frac{N}%
{2}+a\right)  ^{2}+N-1\right)  \int_{\mathbb{R}\times\mathbb{S}^{N-1}%
}\left\vert \mathbf{V}\right\vert ^{2}\mathrm{dtd\sigma}\right] \\
&  \times\left[  \int_{\mathbb{R}\times\mathbb{S}^{N-1}}e^{\left(
2a-2b+2\right)  t}|\mathbf{V}|^{2}\mathrm{dtd\sigma}\right] \\
&  \geq\left(  \sqrt{\left(  1-\frac{N}{2}+a\right)  ^{2}+N-1}+\frac{a-b+1}%
{2}\right)  ^{2}\left[  \int_{\mathbb{R}\times\mathbb{S}^{N-1}}e^{\left(
a-b+1\right)  t}|\mathbf{V}|^{2}\mathrm{dtd\sigma}\right]  ^{2}.
\end{align*}
Indeed, we have
\begin{align*}
&  \int_{\mathbb{R}\times\mathbb{S}^{N-1}}\left\vert \partial_{t}%
\mathbf{V}\left(  t\sigma\right)  +\alpha\mathbf{V}\left(  t\sigma\right)
+\beta e^{\left(  a-b+1\right)  t}\mathbf{V}\left(  t\sigma\right)
\right\vert ^{2}\mathrm{dtd\sigma}\\
&  =\int_{\mathbb{R}\times\mathbb{S}^{N-1}}\left\vert \partial_{t}%
\mathbf{V}\left(  t\sigma\right)  \right\vert ^{2}\mathrm{dtd\sigma}%
+\alpha^{2}\int_{\mathbb{R}\times\mathbb{S}^{N-1}}\left\vert \mathbf{V}\left(
t\sigma\right)  \right\vert ^{2}\mathrm{dtd\sigma}+\beta^{2}\int
_{\mathbb{R}\times\mathbb{S}^{N-1}}e^{\left(  2a-2b+2\right)  t}\left\vert
\mathbf{V}\left(  t\sigma\right)  \right\vert ^{2}\mathrm{dtd\sigma}\\
&  +\alpha\int_{\mathbb{R}\times\mathbb{S}^{N-1}}\partial_{t}\left\vert
\mathbf{V}\left(  t\sigma\right)  \right\vert ^{2}\mathrm{dtd\sigma}+\beta
\int_{\mathbb{R}\times\mathbb{S}^{N-1}}e^{\left(  a-b+1\right)  t}\partial
_{t}\left\vert \mathbf{V}\left(  t\sigma\right)  \right\vert ^{2}%
\mathrm{dtd\sigma}\\
&  +2\alpha\beta\int_{\mathbb{R}\times\mathbb{S}^{N-1}}e^{\left(
a-b+1\right)  t}\left\vert \mathbf{V}\left(  t\sigma\right)  \right\vert
^{2}\mathrm{dtd\sigma}.
\end{align*}
Note that since $\mathbf{U}\in X_{a,b}\left(  \mathbb{R}^{N}\right)  $, we
have
\[
\lim_{r\rightarrow0,\infty}r^{-b+\frac{N}{2}}\mathbf{U}\left(  r\sigma\right)
=\lim_{r\rightarrow0,\infty}r^{-1+\frac{N}{2}-a}\mathbf{U}\left(
r\sigma\right)  =0.
\]
That is,
\[
\lim_{r\rightarrow0,\infty}r^{a-b+1}\mathbf{V}\left(  r\sigma\right)
=\lim_{r\rightarrow0,\infty}\mathbf{V}\left(  r\sigma\right)  =0.
\]
Equivalently
\[
\lim_{t\rightarrow-\infty,\infty}e^{\left(  a-b+1\right)  t}\mathbf{V}\left(
t\sigma\right)  =\lim_{t\rightarrow-\infty,\infty}\mathbf{V}\left(
t\sigma\right)  =0.
\]
Therefore, by integrations by parts, we obtain
\begin{align*}
&  \int_{\mathbb{R}\times\mathbb{S}^{N-1}}\left\vert \partial_{t}%
\mathbf{V}\left(  t\sigma\right)  +\alpha\mathbf{V}\left(  t\sigma\right)
+\beta e^{\left(  a-b+1\right)  t}\mathbf{V}\left(  t\sigma\right)
\right\vert ^{2}\mathrm{dtd\sigma}\\
&  =\int_{\mathbb{R}\times\mathbb{S}^{N-1}}\left\vert \partial_{t}%
\mathbf{V}\left(  t\sigma\right)  \right\vert ^{2}\mathrm{dtd\sigma}%
+\alpha^{2}\int_{\mathbb{R}\times\mathbb{S}^{N-1}}\left\vert \mathbf{V}\left(
t\sigma\right)  \right\vert ^{2}\mathrm{dtd\sigma}+\beta^{2}\int
_{\mathbb{R}\times\mathbb{S}^{N-1}}e^{\left(  2a-2b+2\right)  t}\left\vert
\mathbf{V}\left(  t\sigma\right)  \right\vert ^{2}\mathrm{dtd\sigma}\\
&  +\beta\left(  2\alpha-\left(  a-b+1\right)  \right)  \int_{\mathbb{R}%
\times\mathbb{S}^{N-1}}e^{\left(  a-b+1\right)  t}\left\vert \mathbf{V}\left(
t\sigma\right)  \right\vert ^{2}\mathrm{dtd\sigma}.
\end{align*}

Now, since $\int_{\mathbb{R}\times\mathbb{S}^{N-1}}\left\vert \partial
_{t}\mathbf{V}\left(  t\sigma\right)  +\alpha\mathbf{V}\left(  t\sigma\right)
+\beta e^{\left(  a-b+1\right)  t}\mathbf{V}\left(  t\sigma\right)
\right\vert ^{2}\mathrm{dtd\sigma}\geq0$ for all $\beta\in%
%TCIMACRO{\U{211d} }%
%BeginExpansion
\mathbb{R}
%EndExpansion
$, we deduce that
\begin{align*}
&  \left[  \int_{\mathbb{R}\times\mathbb{S}^{N-1}}\left\vert \partial
_{t}\mathbf{V}\left(  t\sigma\right)  \right\vert ^{2}\mathrm{dtd\sigma
}+\alpha^{2}\int_{\mathbb{R}\times\mathbb{S}^{N-1}}\left\vert \mathbf{V}%
\left(  t\sigma\right)  \right\vert ^{2}\mathrm{dtd\sigma}\right]  \left[
\int_{\mathbb{R}\times\mathbb{S}^{N-1}}e^{\left(  2a-2b+2\right)  t}\left\vert
\mathbf{V}\left(  t\sigma\right)  \right\vert ^{2}\mathrm{dtd\sigma}\right] \\
&  \geq\left(  \alpha-\frac{a-b+1}{2}\right)  ^{2}\left(  \int_{\mathbb{R}%
\times\mathbb{S}^{N-1}}e^{\left(  a-b+1\right)  t}\left\vert \mathbf{V}\left(
t\sigma\right)  \right\vert ^{2}\mathrm{dtd\sigma}\right)  ^{2}.
\end{align*}
If we choose $\alpha^{2}=\left(  1-\frac{N}{2}+a\right)  ^{2}+N-1$, that is
$\alpha=-\sqrt{\left(  1-\frac{N}{2}+a\right)  ^{2}+N-1}$, then the best
constant is%
\[
\left(  \alpha-\frac{a-b+1}{2}\right)  ^{2}=\left(  \sqrt{\left(  1-\frac
{N}{2}+a\right)  ^{2}+N-1}+\frac{a-b+1}{2}\right)  ^{2}\text{.}%
\]
The equality happens when $\partial_{t}\mathbf{V}\left(  t\sigma\right)
+\alpha\mathbf{V}\left(  t\sigma\right)  +\beta e^{\left(  a-b+1\right)
t}\mathbf{V}\left(  t\sigma\right)  =0$ with
\[
\alpha=-\sqrt{\left(  1-\frac{N}{2}+a\right)  ^{2}+N-1}%
\]
and
\[
\beta=\left(  \sqrt{\left(  1-\frac{N}{2}+a\right)  ^{2}+N-1}+\frac{a-b+1}%
{2}\right)  \frac{\int_{\mathbb{R}\times\mathbb{S}^{N-1}}e^{\left(
a-b+1\right)  t}\left\vert \mathbf{V}\left(  t\sigma\right)  \right\vert
^{2}\mathrm{dtd\sigma}}{\int_{\mathbb{R}\times\mathbb{S}^{N-1}}e^{\left(
2a-2b+2\right)  t}\left\vert \mathbf{V}\left(  t\sigma\right)  \right\vert
^{2}\mathrm{dtd\sigma}}>0.
\]
For instance,
\[
\mathbf{V}\left(  t\sigma\right)  =e^{-\alpha t-\frac{\beta}{\left(
a-b+1\right)  }e^{\left(  a-b+1\right)  t}}\sigma\text{.}%
\]
Hence $\mathbf{V}\left(  r\sigma\right)  =r^{-\alpha}e^{-\frac{\beta}{\left(
a-b+1\right)  }r^{\left(  a-b+1\right)  }}\sigma$. That is $\mathbf{U}\left(
r\sigma\right)  =r^{1-\frac{N}{2}+a-\alpha}e^{-\frac{\beta}{\left(
a-b+1\right)  }r^{\left(  a-b+1\right)  }}\sigma$ and $\mathbf{U}\left(
x\right)  =\left\vert x\right\vert ^{-\frac{N}{2}+a-\alpha}e^{-\frac{\beta
}{\left(  a-b+1\right)  }\left\vert x\right\vert ^{\left(  a-b+1\right)  }}x$.
This is curl-free since $\frac{\partial U_{k}}{dx_{j}}=\frac{\partial U_{j}%
}{dx_{k}}$. Note that $\mathbf{U}\in X_{a,b}\left(  \mathbb{R}^{N}\right)  $.
Indeed, it is obvious that $\lim_{r\rightarrow\infty}r^{k}\mathbf{U}\left(
r\sigma\right)  =0$ for all $k$ since $a-b+1>0$ and $\beta>0$. Now,
\begin{align*}
\lim_{r\rightarrow0}r^{-1+\frac{N}{2}-a}\mathbf{U}\left(  r\sigma\right)   &
=\lim_{r\rightarrow0}r^{-1+\frac{N}{2}-a}r^{1-\frac{N}{2}+a-\alpha}%
e^{-\frac{\beta}{\left(  a-b+1\right)  }r^{\left(  a-b+1\right)  }}\sigma\\
&  =\lim_{r\rightarrow0}r^{\sqrt{\left(  1-\frac{N}{2}+a\right)  ^{2}+N-1}}=0
\end{align*}
and
\begin{align*}
\lim_{r\rightarrow0}r^{-b+\frac{N}{2}}\mathbf{U}\left(  r\sigma\right)   &
=\lim_{r\rightarrow0}r^{-b+\frac{N}{2}}r^{1-\frac{N}{2}+a-\alpha}%
e^{-\frac{\beta}{\left(  a-b+1\right)  }r^{\left(  a-b+1\right)  }}\sigma\\
&  =\lim_{r\rightarrow0}r^{a-b+1}r^{\sqrt{\left(  1-\frac{N}{2}+a\right)
^{2}+N-1}}=0.
\end{align*}

\textbf{Case 2:} $a-b+1<0.$

As in Case 1, we get%
\begin{align*}
&  \int_{\mathbb{R}\times\mathbb{S}^{N-1}}\left\vert \partial_{t}%
\mathbf{V}\left(  t\sigma\right)  +\alpha\mathbf{V}\left(  t\sigma\right)
+\beta e^{\left(  a-b+1\right)  t}\mathbf{V}\left(  t\sigma\right)
\right\vert ^{2}\mathrm{dtd\sigma}\\
&  =\int_{\mathbb{R}\times\mathbb{S}^{N-1}}\left\vert \partial_{t}%
\mathbf{V}\left(  t\sigma\right)  \right\vert ^{2}\mathrm{dtd\sigma}%
+\alpha^{2}\int_{\mathbb{R}\times\mathbb{S}^{N-1}}\left\vert \mathbf{V}\left(
t\sigma\right)  \right\vert ^{2}\mathrm{dtd\sigma}+\beta^{2}\int
_{\mathbb{R}\times\mathbb{S}^{N-1}}e^{\left(  2a-2b+2\right)  t}\left\vert
\mathbf{V}\left(  t\sigma\right)  \right\vert ^{2}\mathrm{dtd\sigma}\\
&  +\alpha\int_{\mathbb{R}\times\mathbb{S}^{N-1}}\partial_{t}\left\vert
\mathbf{V}\left(  t\sigma\right)  \right\vert ^{2}\mathrm{dtd\sigma}+\beta
\int_{\mathbb{R}\times\mathbb{S}^{N-1}}e^{\left(  a-b+1\right)  t}\partial
_{t}\left\vert \mathbf{V}\left(  t\sigma\right)  \right\vert ^{2}%
\mathrm{dtd\sigma}\\
&  +2\alpha\beta\int_{\mathbb{R}\times\mathbb{S}^{N-1}}e^{\left(
a-b+1\right)  t}\left\vert \mathbf{V}\left(  t\sigma\right)  \right\vert
^{2}\mathrm{dtd\sigma}.
\end{align*}
Now, since $\mathbf{U}\in X_{a,b}\left(  \mathbb{R}^{N}\right)  $, we have
\[
\lim_{r\rightarrow0,\infty}r^{-1+\frac{N}{2}-a}\mathbf{U}\left(
r\sigma\right)  =\lim_{r\rightarrow0,\infty}r^{-b+\frac{N}{2}}\mathbf{U}%
\left(  r\sigma\right)  =0.
\]
Therefore
\[
\lim_{t\rightarrow-\infty,\infty}e^{\left(  a-b+1\right)  t}\mathbf{V}\left(
t\sigma\right)  =\lim_{t\rightarrow-\infty,\infty}\mathbf{V}\left(
t\sigma\right)  =0.
\]
Therefore, by integrations by parts, we get
\begin{align*}
&  \int_{\mathbb{R}\times\mathbb{S}^{N-1}}\left\vert \partial_{t}%
\mathbf{V}\left(  t\sigma\right)  +\alpha\mathbf{V}\left(  t\sigma\right)
+\beta e^{\left(  a-b+1\right)  t}\mathbf{V}\left(  t\sigma\right)
\right\vert ^{2}\mathrm{dtd\sigma}\\
&  =\int_{\mathbb{R}\times\mathbb{S}^{N-1}}\left\vert \partial_{t}%
\mathbf{V}\left(  t\sigma\right)  \right\vert ^{2}\mathrm{dtd\sigma}%
+\alpha^{2}\int_{\mathbb{R}\times\mathbb{S}^{N-1}}\left\vert \mathbf{V}\left(
t\sigma\right)  \right\vert ^{2}\mathrm{dtd\sigma}+\beta^{2}\int
_{\mathbb{R}\times\mathbb{S}^{N-1}}e^{\left(  2a-2b+2\right)  t}\left\vert
\mathbf{V}\left(  t\sigma\right)  \right\vert ^{2}\mathrm{dtd\sigma}\\
&  +\beta\left(  2\alpha-\left(  a-b+1\right)  \right)  \int_{\mathbb{R}%
\times\mathbb{S}^{N-1}}e^{\left(  a-b+1\right)  t}\left\vert \mathbf{V}\left(
t\sigma\right)  \right\vert ^{2}\mathrm{dtd\sigma}.
\end{align*}
and
\begin{align*}
&  \left[  \int_{\mathbb{R}\times\mathbb{S}^{N-1}}\left\vert \partial
_{t}\mathbf{V}\left(  t\sigma\right)  \right\vert ^{2}\mathrm{dtd\sigma
}+\alpha^{2}\int_{\mathbb{R}\times\mathbb{S}^{N-1}}\left\vert \mathbf{V}%
\left(  t\sigma\right)  \right\vert ^{2}\mathrm{dtd\sigma}\right]  \left[
\int_{\mathbb{R}\times\mathbb{S}^{N-1}}e^{\left(  2a-2b+2\right)  t}\left\vert
\mathbf{V}\left(  t\sigma\right)  \right\vert ^{2}\mathrm{dtd\sigma}\right] \\
&  \geq\left(  \alpha-\frac{a-b+1}{2}\right)  ^{2}\left(  \int_{\mathbb{R}%
\times\mathbb{S}^{N-1}}e^{\left(  a-b+1\right)  t}\left\vert \mathbf{V}\left(
t\sigma\right)  \right\vert ^{2}\mathrm{dtd\sigma}\right)  ^{2}.
\end{align*}
Now, we choose $\alpha=\sqrt{\left(  1-\frac{N}{2}+a\right)  ^{2}+N-1}$ and
get%
\begin{align*}
&  \left[  \int_{\mathbb{R}\times\mathbb{S}^{N-1}}\left\vert \partial
_{t}\mathbf{V}\left(  t\sigma\right)  \right\vert ^{2}\mathrm{dtd\sigma
}+\alpha^{2}\int_{\mathbb{R}\times\mathbb{S}^{N-1}}\left\vert \mathbf{V}%
\left(  t\sigma\right)  \right\vert ^{2}\mathrm{dtd\sigma}\right] \\
&  \times\left[  \int_{\mathbb{R}\times\mathbb{S}^{N-1}}e^{\left(
2a-2b+2\right)  t}\left\vert \mathbf{V}\left(  t\sigma\right)  \right\vert
^{2}\mathrm{dtd\sigma}\right] \\
&  \geq\left(  \sqrt{\left(  1-\frac{N}{2}+a\right)  ^{2}+N-1}-\frac{a-b+1}%
{2}\right)  ^{2}\left[  \int_{\mathbb{R}\times\mathbb{S}^{N-1}}e^{\left(
a-b+1\right)  t}\left\vert \mathbf{V}\left(  t\sigma\right)  \right\vert
^{2}\mathrm{dtd\sigma}\right]  ^{2}\text{.}%
\end{align*}
The equality happens when $\partial_{t}\mathbf{V}\left(  t\sigma\right)
+\alpha\mathbf{V}\left(  t\sigma\right)  +\beta e^{\left(  a-b+1\right)
t}\mathbf{V}\left(  t\sigma\right)  =0$ with
\[
\alpha=\sqrt{\left(  1-\frac{N}{2}+a\right)  ^{2}+N-1}%
\]
and
\[
\beta=\left(  -\sqrt{\left(  1-\frac{N}{2}+a\right)  ^{2}+N-1}+\frac{a-b+1}%
{2}\right)  \frac{\int_{\mathbb{R}\times\mathbb{S}^{N-1}}e^{\left(
a-b+1\right)  t}\left\vert \mathbf{V}\left(  t\sigma\right)  \right\vert
^{2}\mathrm{dtd\sigma}}{\int_{\mathbb{R}\times\mathbb{S}^{N-1}}e^{\left(
2a-2b+2\right)  t}\left\vert \mathbf{V}\left(  t\sigma\right)  \right\vert
^{2}\mathrm{dtd\sigma}}<0.
\]
That is
\[
\mathbf{V}\left(  t\sigma\right)  =e^{-\alpha t-\frac{\beta}{\left(
a-b+1\right)  }e^{\left(  a-b+1\right)  t}}\sigma
\]
and $\mathbf{V}\left(  r\sigma\right)  =r^{-\alpha}e^{-\frac{\beta}{\left(
a-b+1\right)  }r^{\left(  a-b+1\right)  }}\sigma$. Therefore $\mathbf{U}%
\left(  r\sigma\right)  =r^{1-\frac{N}{2}+a-\alpha}e^{-\frac{\beta}{\left(
a-b+1\right)  }r^{\left(  a-b+1\right)  }}\sigma.$ That is $\mathbf{U}\left(
x\right)  =\left\vert x\right\vert ^{-\frac{N}{2}+a-\alpha}e^{-\frac{\beta
}{\left(  a-b+1\right)  }\left\vert x\right\vert ^{\left(  a-b+1\right)  }}x$.
Note that $\mathbf{U}$ is curl-free since $\frac{\partial U_{k}}{dx_{j}}%
=\frac{\partial U_{j}}{dx_{k}}$. Also, it is easy to check that $\mathbf{U}\in
X_{a,b}\left(  \mathbb{R}^{N}\right)  $. Indeed, since $a-b+1<0$, we can see
easily that $\lim_{r\rightarrow0}r^{k}\mathbf{U}\left(  r\sigma\right)  =0$
for all $k$. Now,
\begin{align*}
\lim_{r\rightarrow\infty}r^{-1+\frac{N}{2}-a}\mathbf{U}\left(  r\sigma\right)
&  =\lim_{r\rightarrow\infty}r^{-1+\frac{N}{2}-a}r^{1-\frac{N}{2}+a-\alpha
}e^{-\frac{\beta}{\left(  a-b+1\right)  }r^{\left(  a-b+1\right)  }}\sigma\\
&  =\lim_{r\rightarrow\infty}r^{-\sqrt{\left(  1-\frac{N}{2}+a\right)
^{2}+N-1}}=0
\end{align*}
and%
\begin{align*}
\lim_{r\rightarrow\infty}r^{-b+\frac{N}{2}}\mathbf{U}\left(  r\sigma\right)
&  =\lim_{r\rightarrow\infty}r^{-b+\frac{N}{2}}r^{1-\frac{N}{2}+a-\alpha
}e^{-\frac{\beta}{\left(  a-b+1\right)  }r^{\left(  a-b+1\right)  }}\sigma\\
&  =\lim_{r\rightarrow\infty}r^{a-b+1}r^{-\sqrt{\left(  1-\frac{N}%
{2}+a\right)  ^{2}+N-1}}=0.
\end{align*}

\end{proof}

\section{Weighted second order Heisenberg Uncertainty Principle-Proof of
Theorem \ref{T2}}

\begin{proof}
[Proof of Theorem \ref{T2}]We have for $u\in Y_{a}(\mathbb{R}^{N})$ and any
$s,t\in%
%TCIMACRO{\U{211d} }%
%BeginExpansion
\mathbb{R}
%EndExpansion
$ that
\begin{align*}
&  \int_{\mathbb{R}^{N}}\left\vert \frac{\Delta u}{\left\vert x\right\vert
^{a}}+t\left\vert x\right\vert ^{a+1}\frac{x}{\left\vert x\right\vert }%
\cdot\nabla u+s\left\vert x\right\vert ^{a}u\right\vert ^{2}\mathrm{dx}\\
&  =\int_{\mathbb{R}^{N}}\frac{\left\vert \Delta u\right\vert ^{2}}{\left\vert
x\right\vert ^{2a}}\mathrm{dx}+t^{2}\int_{\mathbb{R}^{N}}\left\vert
x\right\vert ^{2a+2}\left\vert \frac{x}{\left\vert x\right\vert }\cdot\nabla
u\right\vert ^{2}\mathrm{dx}+s^{2}\int_{\mathbb{R}^{N}}\left\vert x\right\vert
^{2a}\left\vert u\right\vert ^{2}\mathrm{dx}\\
&  +2t\int_{\mathbb{R}^{N}}\Delta u\left(  x\cdot\nabla u\right)
\mathrm{dx}+2s\int_{\mathbb{R}^{N}}u\Delta u\mathrm{dx}+2ts\int_{\mathbb{R}%
^{N}}\left\vert x\right\vert ^{2a}u\left(  x\cdot\nabla u\right)
\mathrm{dx}\text{.}%
\end{align*}
Since $u\in Y_{a}(\mathbb{R}^{N})$, we have by integration by parts that
\[
\int_{\mathbb{R}^{N}}\Delta u\left(  x\cdot\nabla u\right)  \mathrm{dx}%
=\frac{N-2}{2}\int_{\mathbb{R}^{N}}\left\vert \nabla u\right\vert
^{2}\mathrm{dx}\text{,}%
\]%
\[
\int_{\mathbb{R}^{N}}u\Delta u\mathrm{dx}=-\int_{\mathbb{R}^{N}}\left\vert
\nabla u\right\vert ^{2}\mathrm{dx}%
\]
and
\begin{align*}
\int_{\mathbb{R}^{N}}\left\vert x\right\vert ^{2a}u\left(  x\cdot\nabla
u\right)  \mathrm{dx} &  =%
%TCIMACRO{\dsum \limits_{j=1}^{N}}%
%BeginExpansion
{\displaystyle\sum\limits_{j=1}^{N}}
%EndExpansion
\int_{\mathbb{R}^{N}}\left\vert x\right\vert ^{2a}ux_{j}\partial
_{j}u\mathrm{dx}\\
&  =-%
%TCIMACRO{\dsum \limits_{j=1}^{N}}%
%BeginExpansion
{\displaystyle\sum\limits_{j=1}^{N}}
%EndExpansion
\int_{\mathbb{R}^{N}}u\partial_{j}\left(  \left\vert x\right\vert ^{2a}%
ux_{j}\right)  \mathrm{dx}\\
&  =-%
%TCIMACRO{\dsum \limits_{j=1}^{N}}%
%BeginExpansion
{\displaystyle\sum\limits_{j=1}^{N}}
%EndExpansion
\left(  \int_{\mathbb{R}^{N}}\left\vert x\right\vert ^{2a}ux_{j}\partial
_{j}u\mathrm{dx}+\int_{\mathbb{R}^{N}}\left\vert x\right\vert ^{2a}\left\vert
u\right\vert ^{2}\mathrm{dx}+\int_{\mathbb{R}^{N}}\left\vert u\right\vert
^{2}2a\left\vert x\right\vert ^{2a-2}x_{j}^{2}\right)  \\
&  =-\int_{\mathbb{R}^{N}}\left\vert x\right\vert ^{2a}u\left(  x\cdot\nabla
u\right)  \mathrm{dx}-N\int_{\mathbb{R}^{N}}\left\vert x\right\vert
^{2a}\left\vert u\right\vert ^{2}\mathrm{dx}-2a\int_{\mathbb{R}^{N}}\left\vert
x\right\vert ^{2a}\left\vert u\right\vert ^{2}\mathrm{dx}\\
&  =-\left(  \frac{N}{2}+a\right)  \int_{\mathbb{R}^{N}}\left\vert
x\right\vert ^{2a}\left\vert u\right\vert ^{2}\mathrm{dx}\text{.}%
\end{align*}
Therefore%
\begin{align*}
&  \int_{\mathbb{R}^{N}}\left\vert \frac{\Delta u}{\left\vert x\right\vert
^{a}}+t\left\vert x\right\vert ^{a+1}\frac{x}{\left\vert x\right\vert }%
\cdot\nabla u+s\left\vert x\right\vert ^{a}u\right\vert ^{2}\\
&  =\int_{\mathbb{R}^{N}}\frac{\left\vert \Delta u\right\vert ^{2}}{\left\vert
x\right\vert ^{2a}}\mathrm{dx}+t^{2}\int_{\mathbb{R}^{N}}\left\vert
x\right\vert ^{2a+2}\left\vert \frac{x}{\left\vert x\right\vert }\cdot\nabla
u\right\vert ^{2}\mathrm{dx}+s^{2}\int_{\mathbb{R}^{N}}\left\vert x\right\vert
^{2a}\left\vert u\right\vert ^{2}\mathrm{dx}\\
&  +t\left(  N-2\right)  \int_{\mathbb{R}^{N}}\left\vert \nabla u\right\vert
^{2}\mathrm{dx}-2s\int_{\mathbb{R}^{N}}\left\vert \nabla u\right\vert
^{2}\mathrm{dx}-2ts\left(  \frac{N}{2}+a\right)  \int_{\mathbb{R}^{N}%
}\left\vert x\right\vert ^{2a}\left\vert u\right\vert ^{2}\mathrm{dx}.%
\end{align*}
By choosing $s=2t\left(  \frac{N}{2}+a\right)  $, we get that for all $t\in%
%TCIMACRO{\U{211d} }%
%BeginExpansion
\mathbb{R}
%EndExpansion
:$%
\begin{equation}
\begin{aligned} & \int_{\mathbb{R}^{N}}\left\vert \Delta u\frac{x}{\left\vert x\right\vert ^{a+1}}+t\left\vert x\right\vert ^{a+1}\nabla u+2t\left( \frac{N}{2}+a\right) u\frac{x}{\left\vert x\right\vert ^{1-a}}\right\vert ^{2}\\ & =t^{2}\int_{\mathbb{R}^{N}}\left\vert x\right\vert ^{2a+2}\left\vert \frac{x}{\left\vert x\right\vert }\cdot\nabla u\right\vert ^{2}\mathrm{dx}-\left( N+2+4a\right) t\int_{\mathbb{R}^{N}}\left\vert \nabla u\right\vert ^{2}\mathrm{dx}+\int_{\mathbb{R}^{N}}\frac{\left\vert \Delta u\right\vert ^{2}}{\left\vert x\right\vert ^{2a}}\mathrm{dx}\text{.} \end{aligned}\label{eq:quadratic-inequality}%
\end{equation}
From here, we deduce that%
\[
\left(  \int_{\mathbb{R}^{N}}\frac{\left\vert \Delta u\right\vert ^{2}%
}{\left\vert x\right\vert ^{2a}}\mathrm{dx}\right)  \left(  \int
_{\mathbb{R}^{N}}\left\vert x\right\vert ^{2a+2}\left\vert \frac{x}{\left\vert
x\right\vert }\cdot\nabla u\right\vert ^{2}\mathrm{dx}\right)  \geq\left(
\frac{N+2+4a}{2}\right)  ^{2}\left(  \int_{\mathbb{R}^{N}}\left\vert \nabla
u\right\vert ^{2}\mathrm{dx}\right)  ^{2}.
\]
Moreover, if the quadratic \eqref{eq:quadratic-inequality} has a real root $t$
for a nontrivial $u$, then
\begin{equation}
t=\frac{N+2+4a}{2}\frac{\int_{\mathbb{R}^{N}}|\nabla u|^{2}dx}{\int
_{\mathbb{R}^{N}}|x|^{2a+2}\left\vert \frac{x}{|x|}\cdot\nabla u\right\vert
^{2}dx},\label{eq:t-is-root}%
\end{equation}
from which we have that $\operatorname{sgn}t=\operatorname{sgn}(N+2+4a)$. 
In particular, such a $u$ must satisfy
\begin{equation*}
\frac{\Delta u}{\left\vert x\right\vert ^{a}}+t\left\vert x\right\vert
^{a+1}\frac{x}{\left\vert x\right\vert }\cdot\nabla u+s\left\vert x\right\vert=0
\end{equation*}
for this given $t$.
We
note that when $N\geq2$, $a+1>0$ implies $N+2+4a>0$ (and hence $N+2+4a<0$
implies $a+1<0$).

We show that $u(x)=\alpha e^{-\frac{\beta}{2(a+1)}|x|^{2(a+1)}}$ is a radial
extremizer. We begin by showing that $u$ belongs to $Y_{a}(\mathbb{R}^{N})$.
First note that, if $\frac{\beta}{a+1}<0$, then $u$ behaves exponentially at
$0$ (when $a+1<0$) and at $\infty$ (when $a+1>0$), and so it must hold that
$\frac{\beta}{a+1}>0$. In case $\frac{\beta}{a+1}>0$, $u$ satisfies the
following limits:
\[
\text{if }a+1>0,\text{ then }u(x)\rightarrow%
\begin{cases}
0 & \text{as }|x|\rightarrow\infty\\
\alpha & \text{as }|x|\rightarrow0
\end{cases}
,
\]%
\[
\text{if }a+1<0,\text{ then }u(x)\rightarrow%
\begin{cases}
0 & \text{as }|x|\rightarrow0\\
\alpha & \text{as }|x|\rightarrow\infty
\end{cases}
,
\]
and when $u$ tends to $0$, it does so exponentially. It follows that
integrability is dictated by the parameter $a$. Supposing $a+1>0$ (and hence
$\beta>0$), it is not hard to see that the integration by parts performed
above for $u(x)=\alpha e^{-\frac{\beta}{2(a+1)}|x|^{2(a+1)}}$ are valid only
when $N+2+4a>0$. Similarly, if $a+1<0$, then there needs to hold $N+2+4a<0$.
(This can be conclude from the computations below.) All that is left to check
is that equality in \eqref{wHUP_curl_scalar} holds for this particular $u$.

To this end, we compute
\begin{align*}
\nabla u  &  =-\alpha\beta e^{-\frac{\beta}{2\left(  1+a\right)  }\left\vert
x\right\vert ^{2\left(  1+a\right)  }}\left\vert x\right\vert ^{1+2a}\frac
{x}{\left\vert x\right\vert },\\
\frac{x}{\left\vert x\right\vert }\cdot\nabla u  &  =-\alpha\beta
e^{-\frac{\beta}{2\left(  1+a\right)  }\left\vert x\right\vert ^{2\left(
1+a\right)  }}\left\vert x\right\vert ^{1+2a},
\end{align*}
and
\begin{align*}
\Delta u  &  =\alpha\beta\left\vert x\right\vert ^{2a}e^{-\frac{\beta
}{2\left(  1+a\right)  }\left\vert x\right\vert ^{2\left(  1+a\right)  }%
}\left(  \beta\left\vert x\right\vert ^{2\left(  1+a\right)  }-2a-1\right) \\
&  -\alpha\beta\left(  N-1\right)  e^{-\frac{\beta}{2\left(  1+a\right)
}\left\vert x\right\vert ^{2\left(  1+a\right)  }}\left\vert x\right\vert
^{2a}.
\end{align*}
Therefore for any $t\in%
%TCIMACRO{\U{211d} }%
%BeginExpansion
\mathbb{R}
%EndExpansion
:$%
\begin{align*}
&  \frac{\Delta u}{\left\vert x\right\vert ^{a}}+t\left\vert x\right\vert
^{a+1}\frac{x}{\left\vert x\right\vert }\cdot\nabla u+2t\left(  \frac{N}%
{2}+a\right)  \left\vert x\right\vert ^{a}u\\
&  =\left[  \beta\left\vert x\right\vert ^{a}\left(  \beta\left\vert
x\right\vert ^{2\left(  1+a\right)  }-2a-1\right)  -\beta\left(  N-1\right)
\left\vert x\right\vert ^{a}-\beta t\left\vert x\right\vert ^{2+3a}+t\left(
N+2a\right)  \left\vert x\right\vert ^{a}\right]  u\\
&  =\left(  \beta-t\right)  \left[  \beta\left\vert x\right\vert
^{2+3a}-\left(  N+2a\right)  \left\vert x\right\vert ^{a}\right]  u\text{.}%
\end{align*}
Therefore, by \eqref{eq:quadratic-inequality}, the quadratic equation
\[
t^{2}\int_{\mathbb{R}^{N}}\left\vert x\right\vert ^{2a+2}\left\vert \nabla
u\right\vert ^{2}\mathrm{dx}-\left(  N+2+4a\right)  t\int_{\mathbb{R}^{N}%
}\left\vert \frac{x}{\left\vert x\right\vert }\cdot\nabla u\right\vert
^{2}\mathrm{dx}+\int_{\mathbb{R}^{N}}\frac{\left\vert \Delta u\right\vert
^{2}}{\left\vert x\right\vert ^{2a}}\mathrm{dx}=0
\]
has exactly one real root $t=\beta$. Note that, by \eqref{eq:t-is-root}, if
$a+1>0$ or if $N+2+4a<0$, then $\frac{\beta}{a+1}>0$. We can now conclude
that
\[
\left(  \int_{\mathbb{R}^{N}}\frac{\left\vert \Delta u\right\vert ^{2}%
}{\left\vert x\right\vert ^{2a}}\mathrm{dx}\right)  \left(  \int
_{\mathbb{R}^{N}}\left\vert x\right\vert ^{2a+2}\left\vert \frac{x}{\left\vert
x\right\vert }\cdot\nabla u\right\vert ^{2}\mathrm{dx}\right)  =\left(
\frac{N+2+4a}{2}\right)  ^{2}\left(  \int_{\mathbb{R}^{N}}\left\vert \nabla
u\right\vert ^{2}\mathrm{dx}\right)  ^{2},
\]
as desired, thereby showing that $u(x)=\alpha e^{-\frac{\beta}{2(a+1)}%
|x|^{2(a+1)}}$ is a radial extremizer.

For $N\geq2$, we will show that there exist infinitely many nonradial
optimizers of (\ref{wHUP_curl_scalar}). Indeed, as observed above, equality happens in (\ref{wHUP_curl_scalar}) if and only if
\begin{equation}
\frac{\Delta u}{\left\vert x\right\vert ^{a}}+t\left\vert x\right\vert
^{a+1}\frac{x}{\left\vert x\right\vert }\cdot\nabla u+s\left\vert x\right\vert
^{a}u=0\label{eq:extremizer-pde}%
\end{equation}
for some $t$ and $s=2t\left(  \frac{N}{2}+a\right)  $ such that
$\operatorname{sgn}t=\operatorname{sgn}(N+2+4a)$.
We will show that \eqref{eq:extremizer-pde} has nonradial solutions.

To begin, we recall Kummer's confluent hypergeometric functions ${}_{1}%
F_{1}(A;B;z)$ which are solutions to Kummer's equation
\begin{equation}
z\frac{d^{2}w}{dz^{2}}+(B-z)\frac{dw}{dz}-Aw=0.\label{eq:kummers-equation}%
\end{equation}
We will show that solutions of \eqref{eq:extremizer-pde} may be expressed in
terms of Kelvin-like transforms of Kummer's confluent hypergeometric
functions, i.e., written in terms of the functions of the form $z^{\alpha}%
{}_{1}F_{1}\left(  A;B;cr^{\beta}\right)  $ for appropriately chosen
$A,B,\alpha,\beta,c\in\mathbb{R}$. To be precise, we show
\begin{align*}
u(x) &  =|x|^{\alpha}{}_{1}F_{1}\left(  \frac{\alpha+N+2a}{2a+2};\frac
{2\alpha+2a+N}{2a+2};-\frac{t}{2a+2}|x|^{2a+2}\right)  ,\\
\alpha &  =\frac{2-N\pm\sqrt{(N-2)^{2}-4\lambda}}{2},
\end{align*}
where $\lambda$ is an eigenvalue of the spherical Laplacian $\Delta_{\sigma}$.
% We begin by
% recalling that the Laplacian on $\mathbb{R}^{N}$ may be written in polar
% coordinates as follows:
% \[
% \Delta=\partial_{rr}+\frac{N-1}{r}\partial_{r}+\frac{1}{r^{2}}\Delta_{\sigma},
% \]
% where $\Delta_{\sigma}$ is the Laplace-Beltrami operator on the standard round
%sphere $\mathbb{S}^{N-1}$ and $r=|x|$.
We will suppose solutions of
\eqref{eq:extremizer-pde} take the form $u(x)=f(r)g(\sigma)$ with
$\Delta_{\sigma}g=\lambda g$ for some $\lambda\in\mathbb{R}$ (this is
justified by writing $u$ in terms of spherical harmonics and therefore
$\lambda=-c_{k}=-k\left(  N+k-2\right)  $, $k=0,1,...$). Then, by using \eqref{eq:spherical-decomp-laplacian},
\eqref{eq:extremizer-pde} may be reformulated as
\[
r^{-a}\left[  f^{\prime\prime}\left(  r\right)  g(\sigma)+\left(  N-1\right)
r^{-1}f^{\prime}\left(  r\right)  g(\sigma)+\lambda r^{-2}f(r)g(\sigma
)\right]  +tr^{a+1}f^{\prime}\left(  r\right)  g(\sigma)+sr^{a}f(r)g(\sigma
)=0.
\]
We thus need to find solutions to the ODE
\begin{equation}
y^{\prime\prime}+\left[  (N-1)r^{-1}+tr^{2a+1}\right]  y^{\prime}+\left[
\lambda r^{-2}+sr^{2a}\right]  y=0.\label{eq:extremizer-ode}%
\end{equation}
Comparing \eqref{eq:extremizer-ode} to \eqref{eq:kummers-equation}, we are led
to take $y=r^{\alpha}v(cr^{\beta})$, where $v(z)={}_{1}F_{1}(A;B;z)$ for some
to be determined $A$ and $B$. We compute%
\begin{align*}
y  & =r^{\alpha}v(cr^{\beta})\\
y^{\prime}  & =\alpha r^{\alpha-1}v(cr^{\beta})+c\beta r^{\alpha+\beta
-1}v^{\prime}\left(  cr^{\beta}\right)  \\
y^{\prime\prime}  & =\alpha\left(  \alpha-1\right)  r^{\alpha-2}v(cr^{\beta
})+(2\alpha+\beta-1)\beta cr^{\alpha+\beta-2}v^{\prime}\left(  cr^{\beta
}\right)  +c^{2}\beta^{2}r^{\alpha+2\beta-2}v^{\prime\prime}\left(  cr^{\beta
}\right)  .
\end{align*}
Therefore, \eqref{eq:extremizer-ode} becomes
\begin{align*}
&  \alpha(\alpha-1)r^{\alpha-2}v(cr^{\beta})+(2\alpha+\beta-1)\beta
cr^{\alpha+\beta-2}v^{\prime}\left(  cr^{\beta}\right)  +c^{2}\beta
^{2}r^{\alpha+2\beta-2}v^{\prime\prime}\left(  cr^{\beta}\right)  \\
&  +\left[  (N-1)r^{-1}+tr^{2a+1}\right]  \left(  \alpha r^{\alpha
-1}v(cr^{\beta})+c\beta r^{\alpha+\beta-1}v^{\prime}\left(  cr^{\beta}\right)
\right)  \\
&  +\left[  \lambda r^{-2}+sr^{2a}\right]  r^{\alpha}v(cr^{\beta})\\
&  =0.
\end{align*}
Simplifying and collecting terms, we obtain
\begin{align*}
&  c^{2}\beta^{2}r^{2\beta-2}v^{\prime\prime}\left(  cr^{\beta}\right)  \\
&  +\left[  2\alpha+\beta+N-2+tr^{2a+2}\right]  \beta cr^{\beta-2}v^{\prime
}\left(  cr^{\beta}\right)  \\
&  +\left[  \left(  \alpha(\alpha+N-2)+\lambda\right)  r^{-2}+(\alpha
t+s)r^{2a}\right]  v(cr^{\beta})=0.
\end{align*}
Letting $\alpha$ solve $\alpha(\alpha+N-2)+\lambda=0$, i.e.,
\[
\alpha=\frac{2-N\pm\sqrt{(N-2)^{2}-4\lambda}}{2},
\]
and further simplifying (recall $s=2t\left(  \frac{N}{2}+a\right)$), we obtain
\begin{equation}
\begin{aligned} &c^{2} \beta^{2} r^{2\beta-2a-2} v''(c r^{\beta})\\ &+\left[ 2\alpha+\beta+N-2 + tr^{2a+2} \right]\beta c r^{\beta-2a-2}v'(c r^{\beta})\\ &+t(\alpha+N+2a)v(c r^{\beta})=0. \end{aligned}\label{eq:penultimate}%
\end{equation}
By choosing
\begin{align*}
\beta &  =2a+2\\
A &  =\frac{\alpha+N+2a}{2a+2}\\
B &  =\frac{2\alpha+2a+N}{2a+2}\\
z &  =-\frac{t}{2a+2}r^{2a+2}\\
c & = -\frac{t}{2a+2},
\end{align*}
we conclude that \eqref{eq:penultimate} is equivalent to
\[
zv^{\prime\prime}(z)+(B-z)v^{\prime}(z)-Av(z)=0,
\]
and so
\[
v(r)={}_{1}F_{1}(A;B;z)={}_{1}F_{1}\left(  \frac{\alpha+N+2a}{2a+2}%
;\frac{2\alpha+2a+N}{2a+2};-\frac{t}{2a+2}r^{2a+2}\right)  ,
\]
with $\alpha$ as above. In conclusion, we obtain as solutions for
\eqref{eq:extremizer-pde} functions of the form
\[
u(x)=|x|^{\alpha}{}_{1}F_{1}\left(  \frac{\alpha+N+2a}{2a+2};\frac
{2\alpha+2a+N}{2a+2};-\frac{t}{2a+2}|x|^{2a+2}\right)  g\left(  \frac
{x}{\left\vert x\right\vert }\right)
\]
where $\Delta_{\sigma}g=\lambda g$ and
\[
\alpha=\frac{2-N\pm\sqrt{(N-2)^{2}-4\lambda}}{2}.
\]
Using superposition, we may obtain more solutions. The radial solutions are
obtained by taking $\lambda=0$, which implies $\alpha=0$ or $2-N$. Indeed, choosing
$\alpha=0$, we conclude
\[
u(x)={}_{1}F_{1}\left(  \frac{N+2a}{2a+2};\frac{N+2a}{2a+2};-\frac{t}%
{2a+2}|x|^{2a+2}\right)  =e^{-\frac{t}{2a+2}|x|^{2a+2}},
\]
which recovers the radial extremizers given above.

If $a+1>0$, then $N+2+4a>N+2a>0$. Therefore, $\frac{t}{2a+2}>0$. We will now
show that $U\in Y_{a}\left(  \mathbb{R}^{N}\right)  $ where
\[
U(x)=|x|^{\alpha}{}_{1}F_{1}\left(  \frac{\alpha+N+2a}{2a+2};\frac
{2\alpha+2a+N}{2a+2};-\frac{t}{2a+2}|x|^{2a+2}\right)  g\left(  \frac
{x}{\left\vert x\right\vert }\right)
\]
and
\[
\alpha=\frac{2-N+\sqrt{(N-2)^{2}-4\lambda}}{2}.
\]
Indeed, using the following asymptotic behavior of Kummer's function of the
first kind (see, for instance, \cite{Buc13}): for $r\rightarrow-\infty$
\[
_{1}F_{1}\left(  a,b,r\right)  \sim\Gamma\left(  b\right)  \frac{\left(
-r\right)  ^{-a}}{\Gamma\left(  b-a\right)  },
\]
we have that as $|x|\rightarrow\infty:$%
\[
_{1}F_{1}\left(  \frac{\alpha+N+2a}{2a+2};\frac{2\alpha+2a+N}{2a+2};-\frac
{t}{2a+2}|x|^{2a+2}\right)  =O\left(  |x|^{-\left(  \alpha+N+2a\right)
}\right)  .
\]
Therefore, as $|x|\rightarrow\infty:$%
\[
U(x)=O\left(  |x|^{-N-2a}\right)  .
\]
Also, using the formula $\frac{d}{dr}\left(  _{1}F_{1}\left(  a,b,r\right)
\right)  =\frac{a}{b}\left(  _{1}F_{1}\left(  a+1,b+1,r\right)  \right)  $
(see \cite{Buc13}), we get
\begin{align*}
\left\vert \frac{x}{|x|}\cdot\nabla U\left(  x\right)  \right\vert  &
=O\left(  |x|^{\alpha-1-\left(  \alpha+N+2a\right)  }\right)  +O\left(
|x|^{\alpha+2a+1-\left(  \alpha+N+2a\right)  -(2a+2)}\right)  \\
&  =O\left(  |x|^{-N-2a-1}\right).
\end{align*}
Hence
\[
\lim_{\left\vert x\right\vert \rightarrow\infty}\left\vert x\right\vert
^{N-1}\left\vert U\left(  x\right)  \right\vert ^{2}=\lim_{\left\vert
x\right\vert \rightarrow\infty}|x|^{-N-4a-1}=0\text{ since }N+4a+1>0\text{,}%
\]%
\[
\lim_{\left\vert x\right\vert \rightarrow\infty}\left\vert x\right\vert
^{N}\left\vert \frac{x}{|x|}\cdot\nabla u\left(  x\right)  \right\vert
^{2}=\lim_{\left\vert x\right\vert \rightarrow\infty}|x|^{-N-4a-2}=0\text{
since }N+4a+2>0,
\]
and
\[
\lim_{\left\vert x\right\vert \rightarrow\infty}\left\vert x\right\vert
^{2a+N}\left\vert U\left(  x\right)  \right\vert ^{2}=\lim_{\left\vert
x\right\vert \rightarrow\infty}|x|^{-N-2a}=0\text{ since }N+2a>0\text{.}%
\]

As $r\rightarrow0$, then
\[
_{1}F_{1}\left(  a,b,r\right)  =1+O\left(  r\right)  .
\]
Therefore, it is easy to see that
\[
\lim_{\left\vert x\right\vert \rightarrow0}\left\vert x\right\vert
^{N-1}\left\vert U\left(  x\right)  \right\vert ^{2}=0
\]%
\[
\lim_{\left\vert x\right\vert \rightarrow0,\infty}\left\vert x\right\vert
^{N}\left\vert \frac{x}{|x|}\cdot\nabla U\left(  x\right)  \right\vert ^{2}=0
\]
and
\[
\lim_{\left\vert x\right\vert \rightarrow0}\left\vert x\right\vert
^{2a+N}\left\vert U\left(  x\right)  \right\vert ^{2}=0.
\]

If $N+2a<0$, then $a+1<0$ and $N+4a+2<0$. Hence, $\frac{t}{2a+2}>0$. In this
case, we choose
\[
\alpha=\frac{2-N-\sqrt{(N-2)^{2}-4\lambda}}{2}.
\]
Now, when $|x|\rightarrow\infty$, then $z=-\frac{t}{2a+2}|x|^{2a+2}%
\rightarrow0$. Hence,
\[
_{1}F_{1}\left(  \frac{\alpha+N+2a}{2a+2};\frac{2\alpha+2a+N}{2a+2};-\frac
{t}{2a+2}|x|^{2a+2}\right)  =1+O\left(  |x|^{2a+2}\right)
\]
and $U(x)=O\left(  |x|^{\alpha}\right)  $. Therefore
\begin{align*}
\left\vert x\right\vert ^{N-1}\left\vert U\left(  x\right)  \right\vert ^{2}
&  =O\left(  |x|^{N-1+2\alpha}\right)  =o\left(  1\right)  \\
\left\vert x\right\vert ^{2a+N}\left\vert U\left(  x\right)  \right\vert ^{2}
&  =O\left(  |x|^{2a+N+2\alpha}\right)  =o\left(  1\right)  \\
\left\vert x\right\vert ^{N}\left\vert \frac{x}{|x|}\cdot\nabla U\left(
x\right)  \right\vert ^{2} &  =O\left(  |x|^{N+2\alpha-2}\right)  =o\left(
1\right)  .
\end{align*}

When $|x|\rightarrow0$, then $z=-\frac{t}{2a+2}|x|^{2a+2}\rightarrow-\infty$.
In this case, as above, we have
\[
\lim_{\left\vert x\right\vert \rightarrow0}\left\vert x\right\vert
^{N-1}\left\vert U\left(  x\right)  \right\vert ^{2}=\lim_{\left\vert
x\right\vert \rightarrow0}|x|^{-N-4a-1}=0\text{ since }N+4a+1<0\text{,}%
\]%
\[
\lim_{\left\vert x\right\vert \rightarrow0}\left\vert x\right\vert
^{N}\left\vert \frac{x}{|x|}\cdot\nabla u\left(  x\right)  \right\vert
^{2}=\lim_{\left\vert x\right\vert \rightarrow0}|x|^{-N-4a-2}=0\text{ since
}N+4a+2<0\text{,}%
\]
and
\[
\lim_{\left\vert x\right\vert \rightarrow0}\left\vert x\right\vert
^{2a+N}\left\vert U\left(  x\right)  \right\vert ^{2}=\lim_{\left\vert
x\right\vert \rightarrow0}|x|^{-N-2a}=0\text{ since }N+2a<0\text{.}%
\]

\end{proof}

\end{document}